\documentclass[12pt,a4paper,reqno]{amsart}

\usepackage{amsmath,amsfonts,amsthm,amssymb,verbatim,enumerate,graphicx}
\usepackage[shortlabels]{enumitem}
\usepackage{transparent}

\usepackage[dvipsnames]{xcolor}
\usepackage{hyperref}

\usepackage[marginratio=1:1,totalwidth=15.75cm,totalheight=22.275cm]{geometry}

\usepackage[alphabetic,nobysame]{amsrefs}

\usepackage{graphicx}
\usepackage[flushmargin]{footmisc}
\usepackage{pdflscape}
\usepackage{marginnote}
\usepackage[all]{xy}
\usepackage{mathrsfs}  

\theoremstyle{plain}
\newtheorem{thm}{Theorem}[section]
\newtheorem{lem}[thm]{Lemma}
\newtheorem{prop}[thm]{Proposition}
\newtheorem{cor}[thm]{Corollary}

\numberwithin{equation}{section}

\newcommand{\id}{\operatorname{id}}

\newcommand{\Deriv}{\mathrm{D}}

\newcommand{\BU}{\mathit{BU}}

\newcommand{\diam}{\operatorname{diam}}

\renewcommand{\subset}{\subseteq}
\renewcommand{\supset}{\supseteq}
\renewcommand{\phi}{\varphi}

\theoremstyle{definition}
\newtheorem{dfn}[thm]{Definition}
\newtheorem{ex}[thm]{Example}

\theoremstyle{remark}
\newtheorem{rmk}[thm]{Remark}

\newcommand{\eps}{\varepsilon}
\newcommand{\interior}{\operatorname{int}}

\newcommand{\DD}{\mathbb{D}}

\newcommand{\dist}{\operatorname{dist}}

\newcommand{\calC}{\mathcal{C}}

\newcommand \C{\mathbb{C}}

\newcommand \Ch{\widehat{\mathbb{C}}}

\newcommand \N{\mathbb{N}}
\newcommand \R{\mathbb{R}}

\newcommand*{\defeq}{\mathrel{\vcenter{\baselineskip0.5ex \lineskiplimit0pt
                     \hbox{\scriptsize.}\hbox{\scriptsize.}}}%
                     =}

\newcommand{\CP}{\operatorname{CP}}
\newcommand{\U}{\mathcal{U}}
\newcommand{\T}{\mathcal{T}}

\usepackage{orcidlink}

\begin{document}

\title[Wandering dynamics]{Wandering dynamics\\ of transcendental functions}
\author[{V. Evdoridou \and D. Mart\'i-Pete \and L. Rempe}]{Vasiliki Evdoridou\,\orcidlink{0000-0002-5409-2663} \and David Mart\'i-Pete\,\orcidlink{0000-0002-0541-8364} \and Lasse Rempe\,\orcidlink{0000-0001-8032-8580}}

\address{School of Mathematics and Statistics\\ The Open University\\ Walton Hall\\Milton Keynes MK7 6AA\\ United Kingdom
%\textsc{\newline \indent \href{https://orcid.org/0000-0002-5409-2663}%{\includegraphics[width=1em,height=1em]{orcid2.png} {\normalfont https://orcid.org/0000-0002-5409-2663}}}%
}
\email{vasiliki.evdoridou@open.ac.uk}

\address{Department of Mathematical Sciences\\ University of Liverpool\\ Liverpool L69 7ZL\\ United Kingdom
%\textsc{\newline \indent \href{https://orcid.org/0000-0002-0541-8364}{\includegraphics[width=1em,height=1em]{orcid2.png} {\normalfont https://orcid.org/0000-0002-0541-8364}}}%
} 
\email{david.marti-pete@liverpool.ac.uk}

\address{Department of Mathematics\\ University of Manchester\\ Manchester M13 9PL\\ United Kingdom
%\textsc{\newline \indent \href{https://orcid.org/0000-0001-8032-8580}{\includegraphics[width=1em,height=1em]{orcid2.png} {\normalfont https://orcid.org/0000-0001-8032-8580}}}%
} 
\email{lasse.rempe@manchester.ac.uk}


\date{\today}

\begin{abstract}
We show that any uniformly escaping and wandering dynamics of a 
 holomorphic function on a compact subset of the plane can  be realised by a transcendental meromorphic function on $\C$.
 
 More precisely, let $\phi$ be a holomorphic function on an open subset of the complex plane,
  and suppose that $K$ is a compact set such that $\phi$ and all its iterates $\phi^n$
  are defined on $K$, and $\phi^n(K)\to\infty$ as $n\to\infty$. We prove that there exist
  a transcendental meromorphic function \smash{$f\colon\C\to\Ch$} 
    and  a compact set \smash{$\widetilde{K}$} such that the dynamics of $f$ on the
  orbit of $\widetilde{K}$ is conjugate, via a smooth change of coordinate close to the identity,
   to that of $\phi$ on the orbit of $K$. If $K$ does not separate the plane, the function $f$ may be chosen to be entire. If all iterates of $\phi$ are univalent on $K$, we can
   take \smash{$\widetilde{K}=K$}.

  We also prove a similar theorem for oscillating dynamics. Finally, we use our 
   results to answer a number of questions of Benini et al. concerning wandering
   domains of entire functions.
\end{abstract}

\maketitle

\section{Introduction}
The study of the iteration of transcendental functions started with the seminal work of Fatou  \cite{fatou26} in 1926. Both Fatou \cite{fatou19} and Julia \cite{julia18} had written extensive memoirs concerning the iteration of rational functions, focusing each on what is now called the Fatou set and the Julia set, respectively. These memoirs became the foundation of the research area known as holomorphic dynamics. When Fatou extended this theory to transcendental entire functions (i.e. holomorphic self-maps of the complex plane for which $\infty$ is an essential singularity), he already discovered some striking differences, including the existence of curves on which points converge quickly to $\infty$ under iteration.

The \textit{Fatou set} $F(f)$ of a transcendental entire or meromorphic function $f$ consists of 
 those points having a neighbourhood on which the iterates $\{f^n\}_{n=0}^{\infty}$ are defined and form a normal family. 
  The \textit{Julia set} $J(f)$ is defined as the complement of $F(f)$ in $\C$; here the behaviour of $f$ under iteration is
  unstable under small perturbations. The Fatou set is open by definition; its connected components are called the \textit{Fatou components} of $f$. The Julia set is either all of $\C$  or it has empty interior. We refer to \cite{bergweiler93} for a survey on the iteration of transcendental meromorphic functions. 

This paper concerns the realisation of wandering dynamics by transcendental functions. Given a transcendental entire or meromorphic function $f$, we say that a compact set $K\subseteq \C$ is a \textit{wandering compact set} of $f$ if $\partial K\subseteq J(f)$ and $f^n(K)\cap f^m(K)=\emptyset $ for all $n\neq m$. If, additionally, the set $K$ is connected, then we say that $K$ is a \textit{wandering continuum}. Note that $J(f)$ always has points with a dense orbit, and these points are examples of (degenerate) wandering continua; in fact, $J(f)$ always contains a wandering Cantor set.

It is an interesting question to ask what other wandering compact sets can arise. \mbox{McMullen} \cite{mcmullen88} showed that for rational maps of the form $f_\lambda(z)=z^2+\lambda/z^3$ with \mbox{$\lambda>0$} sufficiently small, $J(f_\lambda)$ 
contains a Cantor set of wandering Jordan curves. Cheraghi \cite{cheraghi17} 
 proved that the Julia sets of some quadratic polynomials with non-linearisable irrationally indifferent fixed points  contain wandering arcs. To our knowledge, it is not known whether the Julia set of a rational map can contain a non-degenerate
 wandering continuum that is neither an arc nor a Jordan curve. In contrast, it follows from \cite[Theorem~1.6]{arclike} 
 that there exists a transcendental entire function $f$ such that, for every arc-like continuum $X$, the Julia set $J(f)$ 
 contains a wandering continuum that is homeomorphic to $X$. (We refer to~\cite{arclike} for the definition of
 arc-like continua.)

More recently, it was shown in \cite{mrw1,mrw2} that every compact set $K\subset\C$ 
 is a wandering compact set of some transcendental meromorphic function $f$; if $K$ is \textit{full} (that is, $\C\setminus K$ is connected), then $f$ can be chosen to be entire. 
 In these examples, the mapping behaviour of $f$ and its iterates on $K$ is particularly simple. Indeed, $f$ is univalent on some open neighbourhood of $K$, as is every iterate $f^n$ (on a neighbourhood that depends on $n$). The goal of this article is to prove
 a far-reaching generalisation of this result. We show any wandering dynamics of a (not necessarily globally defined) holomorphic
 function can be realised, up to a small smooth change of variable, by a global transcendental entire or meromorphic function,
 provided that the set converges uniformly to $\infty$ under iteration. More precisely, we prove the following theorem. In the following,
   \[ I(f) \defeq \bigl\{z\in\C\colon f^n(z)\to\infty \textup{ as } n\to\infty\bigr\} \]
  denotes the \textit{escaping set} of a transcendental entire or meromorphic function $f$.

\begin{thm}[Realising dynamics by escaping wandering domains] \label{thm:Jordan-continua merged} \label{thm:main-escaping}
Let $(X_n)_{n=0}^\infty$ be a sequence of pairwise disjoint compact sets in $\C$ with $\inf \{|z| \colon z\in X_n\}\to\infty$ as $n\to\infty$ and let
$(C_n)_{n=0}^{\infty}$ be a sequence of finite sets $C_n\subset X_n$. 

Let $\phi$ be a holomorphic function defined on a neighbourhood of $X\defeq\bigcup_{n=0}^\infty X_{n}$ such that $\phi(X_{n-1}) \subset  X_{n}$ and
$\phi(\partial X_{n-1}) \subset \partial X_n$ for all $n\geq 1$. Let $(\varepsilon_n)_{n=0}^\infty$ be a sequence of positive numbers. Then there exist  a transcendental meromorphic function~$f$ and a $\mathcal C^{\infty}$ diffeomorphism $\theta\colon \C\to\C$ such that  
\begin{enumerate}[(a)]
\item $f \circ \theta = \theta \circ \phi$ on $X$;\label{item:functionalrelation}
\item $|\theta(z)-z| \leq \varepsilon_n$ on $X_n$ for all $n\geq 0$;\label{item:conjugacyclose}
\item% 
\label{item:boundaryJulia}%
   $\partial \theta(X)\subseteq J(f)$;
\item $\theta$ satisfies the Cauchy--Riemann equations on $X$; in particular, it is conformal on $\operatorname{int}(X)$;\label{item:conjugacyconformal}
\item $\theta(c)=c$ for $c\in\bigcup_{n=0}^{\infty} C_n$.\label{item:thetaCn}
\end{enumerate}
In particular, $\theta(X)\subseteq I(f)$ and each $\theta(X_n)$ is a wandering compact set of $f$. 
%If $\C\setminus X$ is connected, then $f$ can be chosen to be entire.
\end{thm}

We can also
 realise the dynamics of $\phi$ within the \emph{bungee set} of a transcendental entire or meromorphic function $f$,
  \[ \BU(f) \defeq \bigl\{z\in\C\colon \liminf_{n\to\infty}\lvert f^n(z)\rvert < \infty =
                                   \limsup_{n\to\infty}\lvert f^n(z)\rvert\bigr\}. \]

\begin{thm}[Realising dynamics by oscillating wandering domains]  \label{thm:Jordan-continua merged-oscillating}
Let $(X_n)_{n=0}^\infty$ be a sequence of pairwise disjoint compact sets in $\C$ with $\inf \{|z| \colon z\in X_{n}\}\to\infty$ as $n\to\infty$ and let
$(C_n)_{n=0}^{\infty}$ be a sequence of finite sets $C_n\subset X_n$.

Let $\phi$ be a holomorphic function defined on a neighbourhood of $X\defeq\bigcup_{n=0}^\infty X_{n}$ such that $\phi(X_{n-1}) \subset  X_{n}$ and
$\phi(\partial X_{n-1}) \subset \partial X_{n}$ for all $n\geq 1$. Let $(\varepsilon_n)_{n=0}^\infty$ be a sequence of positive numbers. Then there exist  a transcendental meromorphic function~$f$,  a $\mathcal C^{\infty}$ diffeomorphism $\theta\colon \C\to\C$ and a sequence $(k_n)_{n=0}^{\infty}$ such that  
\begin{enumerate}[(a)]
%\item $\theta(X) = \widetilde{X}$;
\item $f^{k_n} \circ \theta = \theta \circ \phi$ on $X_n$ for all $n\geq 0$;
%\item \lr{the sets \dmp{$f(\theta(X_n))$} are uniformly bounded, i.e.\ $\sup_{n\geq 0} \max_{z\in \theta(X_n)} \lvert f(z)\rvert < \infty$;}
\item $|\theta(z)-z| \leq \varepsilon_n$ on $X_n$ for all $n\geq 0$;
\item %\dmp{every connected component of $\theta(X)$ is a connected component of $J(f)$;}%
$\partial \theta(X)\subseteq J(f)$;
\item $\theta$ satisfies the Cauchy--Riemann equations on $X$; in particular, it is conformal on $\operatorname{int}(X)$;
\item $\theta(c)=c$ for $c\in\bigcup_{n=0}^{\infty} C_n$;\item the set $f(\theta(X))$ is bounded.
\end{enumerate}
In particular, $\theta(X)\subseteq BU(f)$ and each $\theta(X_n)$ is a wandering compact set of $f$.     \end{thm}

\begin{rmk}
The key result to prove Theorems~\ref{thm:Jordan-continua merged} and \ref{thm:Jordan-continua merged-oscillating} is Theorem~\ref{thm:triangle1}, which is a more general result about conjugacies
between non-autonomous holomorphic dynamical systems.
\end{rmk}

\begin{rmk}\label{rmk:Cm-intro}
Let $(M_n)_{n=0}^{\infty}$ be a sequence of non-negative integers.
Then part (b) of Theorems \ref{thm:Jordan-continua merged} and \ref{thm:Jordan-continua merged-oscillating} can be strengthened as follows:  $\theta$ is $\mathcal{C}^{M_n}$-close to the identity on a neighbourhood of $X_n$ and $\theta = \operatorname{id}$ elsewhere.
Moreover, $\theta$ agrees with $\id$ up to order $M_n$ on $C_n$. (See Remark~\ref{rmk:Cm}.)
\end{rmk}

In general, the function $f$ will have poles. Indeed, this is unavoidable, for example,
    if 
    the sets $X_n$ separate the plane, but $X_n$ belongs to the unbounded connected component of $\C\setminus X_{n+1}$ for infinitely many $n\in\N$. However, when
    the sets $X_n$ are non-separating, we can construct $f$ to be entire.

\begin{thm}[The entire case]\label{thm:entire}
 In Theorems~\ref{thm:Jordan-continua merged} and~\ref{thm:Jordan-continua merged-oscillating},
   if $\C\setminus X$ is connected, then $f$ can be chosen to be entire. 
\end{thm}

\begin{ex} \label{ex:rabbit}
Let $K$ be the \emph{Douady rabbit}, which is the filled-in Julia set of $P_c(z)=z^2+c$ with $c\simeq -0.12+0.74i$. This quadratic polynomial has an attracting cycle of period three: $0\mapsto c\mapsto c^2+c\mapsto 0$. For $n\geq 0$, consider 
 the translates $X_n\defeq K+t\cdot n$ of $K$, where $t>0$. If $t$ is sufficiently large, then the $X_n$ are pairwise disjoint. In fact, it is elementary to check that
  $K$ is contained in the closed disc of radius $2$ around the origin, so we may take $t=5$; numerically, one can check that
  $t=3$ also suffices. We can define several model maps $\varphi$ on $X=\bigcup_{n=0}^\infty X_n$ that exhibit different
  dynamics. If we set $\varphi(z)=z+t$ on $X$, then $\phi$ and all its iterates are univalent on $K$. On the other hand, if we define $\varphi_{\vert X_n}(z)=(z-3n)^2+c+3(n+1)$ for $n\geq 0$, then we obtain a map which, up to translation, has the same dynamics as $P_c$ (see Figure~\ref{fig:rabbit}).
\end{ex}

\begin{rmk}
In Theorems~\ref{thm:main-escaping} and~\ref{thm:Jordan-continua merged-oscillating}, 
part~\ref{item:boundaryJulia} can be strengthened as follows:
every connected component of $\theta(\partial X)$ is a connected component of $J(f)$. (The function we construct to achieve this will have poles even if $\C\setminus X$ is connected.)
\end{rmk}

\begin{figure} \label{fig:rabbit}
\def\svgwidth{\textwidth}
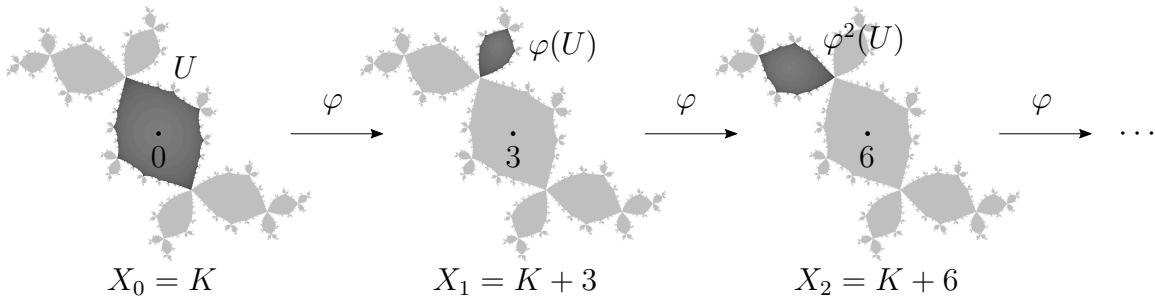
\caption{The continua $(X_n)_{n=0}^\infty$ and model map $\varphi$ from Example~\ref{ex:rabbit}. Here $K$ is Douady's rabbit (in light gray), and the domain~$U$ (in dark gray) is one of the components of $\textup{int}(K)$. Note that $\varphi^{n+3}(U)=\varphi^n(U)+9$ for all $n\geq 0$.}
\end{figure}

In our theorems, the set $\theta(X_0)$ will
 differ from $X_0$ (although they are related by
 a diffeomorphism close to the identity). This appears to be
 difficult to avoid in general. However, when $\varphi$ is injective, we show
 that we can realise $X_0$ exactly, recovering  the above-mentioned previous
 results~\cite[Theorem~1]{bocthaler21}, \cite[Theorem~1.7]{mrw1} and~\cite[Theorem~1.3]{mrw2} as special cases.

\begin{thm}[Realising univalent dynamics]\label{thm:injective}
Under the hypotheses of Theorems \ref{thm:main-escaping}, \ref{thm:Jordan-continua merged-oscillating} or~\ref{thm:entire}, suppose additionally that 
$\varphi$ is univalent on a neighbourhood of $X$ and that $\phi(X_n)=X_{n+1}$
for all $n\geq 0$. Then 
we can choose $\theta$ so that $\theta = \id$ on $X_0$. 
\end{thm}

\subsection*{Applications}
Wandering compact sets are closely related to \emph{wandering domains}: Fatou components that are not eventually periodic.
 If $K$ is a wandering compact set of~$f$, then every component of $\textup{int}(K)$ is a wandering domain of~$f$; hence, in Theorem~\ref{thm:main-escaping}, every interior component of $\theta(X)$ is a wandering domain of~$f$. Note that the converse is not necessarily true: if $U$ is a wandering domain of~$f$, then $\overline{U}$ may not be a wandering continuum of~$f$ in the
 sense of our definition. Indeed, $f(z)=z+\sin(z)+2\pi$ has wandering domains $U$ along the real line which satisfy that $\overline{U}\cap f(\overline{U})\neq\emptyset$. In this example, the closure of each wandering domain is a wandering continuum for $f^2$; however, in general it is not known, for example, whether the boundary of a wandering domain can contain a fixed point of $f$.

Wandering domains can be classified into \textit{escaping}, \textit{oscillating} or \textit{bounded-orbit} according to whether points in the domain belong to
$I(f)$, to $\BU(f)$, or have bounded orbits, respectively. Whether bounded-orbit wandering domains exist is a well-known open problem~\cite[Problem~2.67]{hayman-lingham19}.
In Theorem~\ref{thm:Jordan-continua merged}, any interior component of $X$ is
an escaping wandering domain, while in Theorem~\ref{thm:Jordan-continua merged-oscillating},
they are oscillating wandering domains.

We now describe how to use our theorems to answer a number of open questions
 concerning the possible behaviour of entire functions on simply connected 
 wandering domains. Benini et al.~\cite{benini-fagella-evdoridou-rippon-stallard} 
 proposed a classification of the \emph{internal dynamics} of such wandering 
 domains. On the one hand, they describe this dynamics as 
 \emph{contracting}, \emph{semi-contracting} or \emph{eventually isometric},
 depending on the eventual behaviour of iterates with respect to the
 hyperbolic metric. They also distinguish whether orbits 
 \emph{stay away} from the boundary (with respect to Euclidean distance), 
 \emph{converge} to the boundary, or  display both types of behaviour along different subsequences. (See Section~\ref{sec:topology-wandering-domains} for precise definitions.) These two descriptions give rise to 
 nine possible types of escaping wandering domains, and six possible
 types of oscillating wandering domains (where the iterates cannot stay away from the boundary).

Benini et al. also show that all nine types of escaping wandering domains
can be realised~\cite[Theorem~E]{benini-fagella-evdoridou-rippon-stallard}, while Evdoridou, Rippon and Stallard~\cite[Theorem~1.1]{evdoridou-rippon-stallard} construct
 examples of all six types of oscillating wandering domains. For some of the
 escaping examples, 
 it is shown in~\cite[Theorem~7.2]{benini-fagella-evdoridou-rippon-stallard}
 that the wandering domains can be chosen to be Jordan domains, but for the
 remaining types, and for oscillating wandering domains, it remained open whether
 this is possible.

We may answer this question using our realisation theorems. Indeed, 
 the above-mentioned examples are constructed by approximating 
 appropriate sequences of Blaschke products, and we obtain the following.

 \begin{cor} \label{cor:Jordan}
 For each of the nine (resp.\ six) types of internal dynamics of
   escaping (resp.\ oscillating) internal dynamics, there exists 
   a transcendental entire function having a wandering domain $U$ with this
   type of internal dynamics and such that $\partial U$ is a $\calC^{\infty}$ Jordan curve. In the case of eventually isometric internal dynamics, we may in fact  take 
   $U=\DD$.
\end{cor}

As far as we are aware, these are the first examples of 
   contracting or semi-contracting wandering domains with smooth boundaries.

We also use our results to answer a number of questions about the boundary 
 behaviour of wandering domains raised in~\cite{benini-fagella-evdoridou-rippon-stallard2}. The 
authors study the boundary behaviour of sequences of holomorphic functions on simply connected domains. In analogy to the well-known notion of the Denjoy--Wolff point, they define the Denjoy--Wolff set on the boundary of such domains; see Definition \ref{DWset}. In the same paper, they construct a variety of examples of sequences of finite Blaschke products with intriguing boundary behaviour. For example,~\cite[Theorem~D]{benini-fagella-evdoridou-rippon-stallard2} shows that it is possible that all interior orbits tend to 1 but all boundary orbits eventually leave an arc around~1. The authors 
raise the question~\cite[Remark~2 after Theorem~D]{benini-fagella-evdoridou-rippon-stallard2} whether such an example can be realised by a wandering domain of an
entire function; we show that this is indeed the case (see Example~\ref{ex:emptyDW}). 
In fact, we show that all the examples constructed in~\cite[Section~8]{benini-fagella-evdoridou-rippon-stallard2} can be realised in this manner, answering the question
posed on \cite[p.~7]{benini-fagella-evdoridou-rippon-stallard2}.

We mention four further applications of our results.
\begin{enumerate}[(a)]
    \item In~\cite{EGM}, Theorem~\ref{thm:entire} is used to construct 
      several examples of entire functions with wandering domains 
      that have interesting properties from the point of view of
      universality of composition operators. 
    \item Our result provides a new proof of \cite[Theorem~B]{evdoridou-fagella-geyer-pardo-simon}, which constructs a wandering domain exhibiting both
     discrete and indiscrete orbit relations.
     \item In~\cite{gustavo-lasse}, Theorem~\ref{thm:main-escaping} is used to construct a new type of multiply connected wandering domain
     for meromorphic functions: a \emph{locally but not globally eventually isometric} wandering domain; we refer to~\cite{gustavo-lasse} for definitions.
     \item In~\cite{mrw3}, our technique is combined with a construction from~\cite{mrw1}
       to answer a question of Bishop~\cite[p.~133]{bishop14} from 2014 (see also \cite[Question~6.47]{bergweiler-rempe-2025}) concerning non-escaping points
       on the boundaries of escaping wandering domains.
\end{enumerate}

\subsection*{Notation} We denote by $D(z,r)$ the disc of centre $z\in\C$ and radius $r>0$; we write $\mathbb{D}$ for the unit disc $D(0,1)$. Throughout the paper, for a set $X\subseteq \C$, the boundary $\partial X$ and closure $\overline{X}$ are taken in $\mathbb{C}$. If $X_1,X_2\subseteq \C$ are closed sets, then $\textup{dist}(X_1,X_2)$ denotes the \textit{Euclidean distance} between them. If $U\subseteq \C$ is a domain and $z,w\in U$, we denote by 
$\operatorname{dist}_U(z,w)$ the \textit{hyperbolic distance} between $z$ and $w$ in $U$. If $U,V\subseteq \C$ are open, then we write $U\Subset V$, 
 and say that $U$ is \textit{compactly contained} in $V$, if $U$ is bounded and $\overline{U}\subseteq V$. For $0\leq n\leq \infty$, the class $\mathcal C^n=\mathcal{C}^n(\R^2)$ consists of the functions $f:\R^2\to\R^2$ that are $n$ times differentiable with continuous $n$th partial derivatives, and can be endowed with the norm
\[\|f\|_{\mathcal C^n}\defeq\sum_{k=0}^n \sup_{(x,y)\in\R^2} \|\Deriv^k f(x,y)\|, \quad \textup{for } 0\leq n<\infty,\]
where $\|\cdot \|$ denotes the operator norm.
%and, in the $\mathcal C^\infty$ case, with the norm
%\[\|f\|_{\mathcal C^\infty}\defeq\sum_{k=0}^\infty \frac{1}{2^k}\frac{\sup_{(x,y)\in\R^2} |f^{(k)}(x,y)|}%{1+\sup_{(x,y)\in\R^2} |f^{(k)}(x,y)|}.\]
Whenever there is no possibility of confusion, we will just write $\|f\|$ instead of $\|f\|_{\mathcal C^0}$. For a holomorphic function $\varphi:U\to \C$,  $\textup{CP}(\varphi)$ denotes the set of all critical points of $\varphi$ in $U$.

\subsection*{Structure of the paper} In Section~\ref{sec:prelim}, we introduce the tools that we require from approximation theory. In Section~\ref{sec:conjugacy} we prove that two holomorphic maps defined on a compact set that are close and have the same critical points and critical values are conjugate by a conformal isomorphism which extends to $\C$. Section~\ref{sec:conjugacies-nonautseq} contains the key result used in the proofs of Theorems~\ref{thm:main-escaping} (escaping dynamics), \ref{thm:Jordan-continua merged-oscillating} (oscillating dynamics) and \ref{thm:entire} (the entire case), which are in Section~\ref{sec:main-proof}. The univalent case is addressed in Section~\ref{sec:univalent}. In Section~\ref{sec:topology-wandering-domains} we construct examples of wandering domains realising all kinds of internal dynamics while having Jordan boundaries. Finally,  Section~\ref{sec:examples} contains several applications of our results to construct wandering domains with interesting boundary dynamics.

\subsection*{Acknowledgements} We thank Gustavo Ferreira, Phil Rippon, Gwyneth Stallard and James Waterman for interesting discussions on this topic.

\section{Preliminaries} \label{sec:prelim}
%\subsection{Approximation}

In this section we introduce the tools we need from approximation theory. We will use the following version of Runge's theorem \cite[Theorem 2 on p. 94]{gaier87}, which is based on \cite[Main~lemma]{eremenko-lyubich87}, and it can be proved in a similar way; see also \cite[Appendix]{bocthaler21} and \cite[Theorem~4.1]{mrw2}. 

\begin{thm}[Extension of Runge's theorem]
 \label{thm:runge}
Let $A\subseteq \C$ be a compact set.  Suppose that $g\colon A\to \C$ is meromorphic on  a neighbourhood of $A$. 
  Then for every $\eps>0$, $M\in \mathbb{N}$ and every finite collection of points $C=\{z_n\}_{n=1}^{N}\subseteq A$, there exists a rational function~$f$ such that 
\[
\lvert f(z) - g(z)\rvert < \eps\quad \text{for all } z\in A \setminus g^{-1}({\infty})
\]
and  
\[
f^{(m)}(z_n)=g^{(m)}(z_n)
\]
for all $0\leq m\leq M$ and $1\leq n\leq N$. If $A$ is full and $g$ is holomorphic on a neighbourhood of $A$, then $f$ can be chosen to be a polynomial.
\end{thm}

The following lemma will be used later to ensure that the boundaries of our compact sets are a subset of the Julia set.

 \begin{lem} \label{lem:preimages-points}
Let $U,V\subseteq \C$ be open and $f\colon U\to V$ be holomorphic and nowhere locally constant. Suppose that $X\subseteq U$ is a compact set. For every $\varepsilon>0$, there is $\delta>0$ with the following property. 
Suppose that $P\subseteq \C$ satisfies $\operatorname{dist} (z, P) \leq \delta$ for every $z\in f(X)$, and
 suppose that $\tilde{f}\colon U\to\C$ is a holomorphic function with  
  $|f-\tilde{f}|<\delta$ on $U$. Then $\operatorname{dist}(w,\tilde{f}^{-1}(P))\leq \varepsilon$ for every $w\in X$. 
 \end{lem}
\begin{proof} 
We first prove the lemma for $\tilde{f}=f$. Let $(B_n)_{n\in N}$ be a finite collection of open balls of radius $\varepsilon/2$, with $B_n$ centred at $z_n\in X$, such that $X\subset \bigcup_{n\geq 0} B_n$. By the open mapping theorem, the image $f(B_n)$ is an open set containing $f(z_n)$. Let $r_n>0$ be the radius of the largest ball of centre $f(z_n)$ that is contained in $f(B_n)$. Consider 
\[ \delta= \min_{n \in N} r_n .\]
Let $P$ be such that $\operatorname{dist} (z, P) \leq \delta$ for all $z\in f(X)$, then $f^{-1}(P)$ has a point in each ball $B_n$, so    $\operatorname{dist}(w,f^{-1}(P))\leq \varepsilon$ for every $w\in X$. 

The case where $\tilde{f}$ is close to $f$ follows from the previous one by Hurwitz's theorem. \end{proof}

%\subsection{Other preliminaries}
%The following is a version of the well-known Hurwitz theorem, which we use in the next section.
%\begin{thm}\label{thm:Hurwitz}
%Let $({f_k})$ be a sequence of holomorphic functions on a connected open set $G$ that converge uniformly on compact subsets of $G$ to a holomorphic function $f$ which is not constantly zero on G. If $f$ has a zero of order $m$ at $z_0$ then for every small enough $\rho > 0$ and for sufficiently large $k \in \mathbb{N}$ $ f_k$ has precisely $m$ zeros in the disc defined by $|z - z_0| < \rho$, including multiplicity.
%\end{thm}

%\begin{lem} \label{lem:wandering-Cantor-set}
%Let $f$ be a transcendental meromorphic function. Then $J(f)$ contains a wandering Cantor set.
%\end{lem}
%\begin{proof}
%We will construct an infinite iterated function system. Let $z_0\in J(f)$ and assume $U_0\subseteq \C$ is an open neighbourhood of $z_0$. Then, by the blowing-up property of the Julia set, there exists $n_0\in\N$ so that $f^{n_0}(U_0)$ contains $\overline{U_0}$ and $(f^{n_0}(U_0)\setminus \overline{U_0})\cap J(f)\neq \emptyset$. Set $U=f^{n_0}(U_0)$ and let $z_1\in (U\setminus \overline{U_0})\cap J(f)$. By the blowing-up property, there exists a neighbourhood $U_1$ of $z_1$ so that $f^{n_1}(U_1)=U$ and $(U\setminus (\overline{U_0}\cup \overline{U_1}))\cap J(f)\neq \emptyset$. Proceeding in this way, we can obtain infinitely many domains $(U_n)_{n=0}^\infty\subseteq U$ such that $f^{k_n}(U_n)=U$. 
%\end{proof}

We also make use of two simple results concerning approximation. 
   The first is a special case of \cite[Lemma 2.3]{mrw1} and concerns
   approximations of a conformal isomorphism. (It also follows
   from Lemma~\ref{lem:theta-map} below.)
 
 \begin{lem}\label{lem:univ-approx}
   Let $U,V\subseteq\C$ be open, and let $\phi\colon U\to V$ be a conformal isomorphism.
   
Let $A\subseteq U$ be a compact set. Then there is $\eps>0$ with the following property: if $f\colon U\to\C$ is holomorphic with $\lvert f(z) - \phi(z)\rvert \leq\eps$ for all $z\in U$, then 
      $f$ is injective on $A$, with $f(A)\subseteq V$. 
 \end{lem} 

The second result states, essentially, that composition of functions is a continuous
operation; we omit the simple proof. (Compare also~\cite[Lemma 2.4]{mrw1}.)

\begin{lem}[Approximation of compositions]\label{lem:approx-iterates}
 Let $n\geq 0$, let $(U_k)_{k=0}^n$ be a sequence of
 open subsets of $\C$ and let $(g_k)_{k=0}^{n}$ be a sequence of
  continuous functions $g_k\colon U_k\to\C$. For $k\leq n$, define 
   \[ G_k\defeq g_k \circ \dots\circ g_0.\]
   Suppose that 
    $K\subseteq U_0$ is compact such that $G_n$ is defined on $K$.

    Then for every $\eps>0$, there is $\delta>0$ 
     with the following property. If $(f_k)_{k=0}^n$ are continuous functions
     $f_k\colon U_k\to\C$ 
     with $\lvert f_k(z) - g_k(z)\rvert < \delta$ for all $z\in U_k$, then
      $F_n$ is defined on $K$, where 
   \[ F_k\defeq f_k \circ \dots\circ f_0\]
     for $k\leq n$, and 
       \begin{equation}\label{eqn:iteratesclose} \lvert F_k(z) - G_k(z)\rvert < \eps   
       \end{equation}
       for all $z\in K$ and all $k\leq n$.       
\end{lem}

\section{Approximation of holomorphic maps}
\label{sec:conjugacy}

In this section we prove that if two holomorphic functions are close and their critical points are also close (and of the same multiplicity), then the two functions differ by precomposition with a homeomorphism which is close to the identity. Note that two holomorphic functions can be arbitrarily close but have different sets of critical points (e.g. $f(z)= z^3$ has a double critical point but $g(z)=z^3+\varepsilon$ has two simple critical points for $\varepsilon\neq 0$).

\begin{lem}\label{lem:theta-map}
Let $U \subseteq \C$ be open, let $\varphi\colon U \to \mathbb{C}$ be a  holomorphic function. Also let 
$V\Subset U$ be open and assume that 
there exists a finite set 
$C\subset V$ containing all critical points of $\varphi$ in $U$. For every $\varepsilon >0$ there exists $\delta>0$ with the following property.

Suppose $f\colon U \to \C$ is holomorphic such that $|f-\varphi|\leq\delta$ on $U$,
such that $f=\varphi$ on $C$, and such that every critical point of
$\varphi$  in $V$ is a critical point of $f$ of the same multiplicity. 
%and such that for every critical point~$c$ of $\varphi$ in $X$, there exists a critical point $\tilde{c}$ of $f$ of the same multiplicity with $|c-\tilde{c}|<\delta$ and $\phi(c)=f(\tilde{c})$. 
Then there exists a conformal isomorphism $\theta\colon V \to \theta (V)$ such that  $f\circ \theta= \varphi$ on $V$, such that $|\theta(z)-z|<\varepsilon$ for all $z \in V$, and such that
$\theta =\id$ on $C$.
\end{lem}

\begin{proof} 
First we assume that $C=\emptyset$, and in particular that $\varphi$ has no critical points in $U$. Then $\varphi$ is locally univalent.

Let $\eta>0$ be so small  that 
$\eta < \operatorname{min} 
\bigl(\operatorname{dist}(\overline{V}, \partial U)/2, \varepsilon\bigr)$ and 
\begin{equation}\label{eqn:eta} 2\eta < \min\bigl\{\lvert \zeta - z\rvert \colon       \zeta,z\in \overline{V}, \zeta\neq z
\text{ and }\phi(\zeta)=\phi(z)\bigr\}. \end{equation}
The minimum exists because
 $\phi$ is continuous and $\overline{V}$ is
 a compact subset of $U$; it is positive
 because $\phi$ is locally univalent.
Finally, choose
 \[ \delta < \min\bigl\{\lvert \phi(z) - \phi(\zeta)\rvert \colon 
    z\in\overline{V},\zeta\in U\text{ and }\lvert z - \zeta\rvert = \eta \bigr\}/2.\]
    
Now suppose that $f$ is as in the statement
of the lemma and define $\theta\colon V\to\C$ 
by
\[\theta(z)\defeq\frac{1}{2\pi i}\int_{\partial D(z,\eta)} \zeta \frac{f'(\zeta)}{f(\zeta)-\varphi({z})}\mathrm{d}\zeta.
\]
To see that $\theta$ is indeed defined for 
all $z\in V$, note that 
$\overline{D(z,\eta)}\subset U$ by choice of $\eta$ and  
 \[ \lvert f(\zeta) - \varphi(z)\rvert 
     \geq \lvert \varphi(\zeta) - \varphi(z)\rvert - \lvert f(\zeta) - \varphi(\zeta)\rvert > 2\delta - \delta = \delta \]
 when $\lvert \zeta - z\rvert = \eta$.
 Furthermore,  $\theta$ is 
 a holomorphic function of $z$ since the 
 integrand is holomorphic in $z$.

Let $z\in V$. Then 
$\zeta\mapsto \phi(\zeta)-\phi(z)$ has exactly
one root in $D(z,\eta)$ by choice of $\eta$,
and this root is simple since $\phi$ is
locally univalent. If $\zeta \in \partial D(z,\eta)$, then
  \[ \lvert f(\zeta) - \varphi(\zeta)\rvert 
      \leq \delta < \lvert \varphi(\zeta) - \varphi(z)\rvert. \]
 By  Rouch\'e's theorem, $\zeta \mapsto f(\zeta)- \varphi(z)$ also has exactly one root $\alpha$ in $D(z, \eta)$, counting multiplicities. By the residue theorem,
\begin{align*}\theta(z)=\textup{Res}\Bigl(\zeta\mapsto \zeta\frac{f'(\zeta)}{f(\zeta)-\varphi(z)},\alpha\Bigr)&=
\lim_{\zeta\to\alpha} (\zeta-\alpha) \zeta\frac{f'(\zeta)}{f(\zeta)-\phi(z)} \\ 
&=\lim_{\zeta\to\alpha} \zeta\cdot 
  \frac{\zeta-\alpha}{f(\zeta)-f(\alpha)}\cdot
  f'(\zeta) = \alpha\cdot \frac{f'(\alpha)}{f'(\alpha)} = \alpha.
\end{align*}
Hence $f(\theta(z)) = \phi(z)$ for all 
    $z\in V$. Furthermore, 
$\theta(z)=\alpha \in D(z,\eta)$ by construction,
and hence 
$|\theta (z)-z| < \eta < \varepsilon$
for all $z\in V$. For the same reason, if
$\theta(z)=\theta(w)$, then 
$\lvert z - w\rvert < 2\eta$. 
 By~\eqref{eqn:eta}, we have $z=w$. So
 $\phi$ is injective, and hence a conformal
 isomorphism onto its image. This concludes
 the proof in the case where $C=\emptyset$.

Now assume that $C\neq\emptyset$.
For every $c\in C$, let $d_c$ be the local
 degree of $\phi$ at $c$. We may choose
 a small disc $W(c)$ around $\phi(c)$ such that
 $\overline{W(c)}\cap \phi(C) = \{\phi(c)\}$ and such that
the component $\Omega(c)$ of $\phi^{-1}(W(c))$ that contains $c$ 
satisfies $\overline{\Omega(c)}\subset V$. 
 Then $\Omega(c)$ is a Jordan domain and 
 $\phi\colon \Omega(c)\to W(c)$ is a branched covering
 map, with no critical points except
 (possibly) at $c$. 
 Choose $\rho>0$  so small that
   $\overline{D(c,\rho)}\subset \Omega(c)$ for all $c\in C$.

Now define $U^*\defeq U\setminus C$ and 
\[ V^*\defeq V\setminus 
    \bigcup_{c\in C} \overline{D(c,\rho)}
    \Subset U^*.\]
 Then $\phi$ has no critical points
  in $U^*$. Let $\delta>0$ be obtained
  by applying the case $C=\emptyset$ to the map 
  $\phi|_{U^*}:U^*\to\C$ with $V^*\Subset U^*$ and $\eps^* = \min(\eps,\dist(V^*,\partial U^*)/2)$.

Suppose that  $f$ is as in the statement of the 
lemma. By choice of $\delta$, there is a 
conformal isomorphism $\theta^*\colon V^* \to \theta (V^*)$ such that
 $f(\theta^*(z)) = \phi(z)$  and
 $\lvert \theta^*(z) - z\rvert < \eps^*$
 for all $z\in V^*$.

We claim that $\theta^*$ extends analytically
to $V$. Let $c\in C$; by assumption,
 $f$ has local degree $d_c$ at $c$. Moreover,
by choice of $W(c)$, the map
$\phi$ is a degree $d_c$ covering map when restricted to the Jordan curve $\partial \Omega(c)$. 
Since $f(\theta(z)) = \phi (z)$ on $\partial \Omega(c)$, the map $f$ is $d_c$-to-1  on $\theta(\partial \Omega(c))$. Let $\widetilde{\Omega}(c)$ be
the Jordan domain surrounded by 
 $\theta(\partial \Omega(c))$; by choice of
 $\eps^*$, we have $\widetilde{\Omega}(c)\subset U$ and
 $c\in \widetilde{\Omega}(c)$. 
 It follows by the Riemann--Hurwitz formula that $c$ is the only critical point
 of $f$ in $\widetilde{\Omega}(c)$.

So both $\phi\colon \Omega(c)\to W(c)$
and $f\colon \widetilde{\Omega}(c)\to W(c)$ are branched covering maps of degree $d_c$ with 
a single critical point at $c$. Hence there
is a conformal isomorphism
$\theta_c\colon \Omega(c)\to \widetilde{\Omega}(c)$ with
$\theta_c(c)=c$ and $f\circ \theta_c = \phi$, which may be chosen to agree
with $\theta^*$ on $\partial \Omega(c)$. The desired
map $\theta$ is given by
\[ \theta(z)\defeq \begin{cases}
            \theta^*(z) &\text{if }z\in V^*;\\
            \theta_c(z) &\text{if }z\in \Omega(c),
            c\in C. \end{cases}\]

 By definition, $f\circ \theta = \phi$ and
  $\theta(c)=c$ for all $c\in C$. Furthermore,
  $\lvert \theta(z) - z\rvert <\eps^* \leq \eps$ on $V^*$. Since 
  \[ V\setminus V^* = 
     \bigcup_{c\in C} \overline{D(c,\rho)}
     \Subset V,\]
     we also have 
  $\lvert \theta(z)-z\rvert <\eps $ on $V$
  by the maximum principle. That $\theta$ is 
  a conformal isomorphism follows from
  the definition. This completes
  the proof of the lemma.
\end{proof}

\begin{figure}[t]
\[
\xymatrix{
V_0 \ar[d]_{\theta_0} \ar[r]^{g_0} & V_1 \ar[d]_{\theta_1} \ar[r]^{g_1} %\ar[d]_{\theta_2}
&V_2 \ar[d]_{\theta_2}  \ar[r]^{g_2}& \quad \dots&  \ar[r]%%^{g_{\ell-1}}  
& V_{\ell} \ar[dr]^{g_{\ell}} \ar[d]_{\theta_{\ell}}\\
\theta_0(V_0)  \ar[r]^{f_0} & \theta_1(V_1) \ar[r]^{f_1}& \theta_2(V_2) \ar[r]^{f_2} & \quad \dots & \ar[r]%%^{f_{\ell-1}}
& \theta_{\ell}(V_\ell)\ar[r]^{f_{\ell}} & \C\\
}
\]
\caption{Diagram for the proof of Lemma {\ref{lem:theta-n}}.} \label{fig:row}
\end{figure}
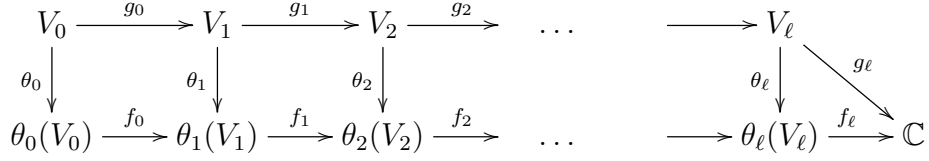

%\hspace*{2cm}\includegraphics[scale=0.5]{diagram.png}
\begin{lem}\label{lem:theta-n}
Let $U_0,\dots,U_{\ell}\subset\C$ be open sets, and let $g_n$ be a holomorphic function on $U_{n}$ for $n=0,\dots,\ell$.
  Also let $V_n\Subset U_n$ be open for $n=0,\dots,\ell$ such that
  $g_{n}(\overline{V_n})\subset V_{n+1}$ for $n=0,\dots,\ell-1$. Furthermore, let
  $C_n\subset V_n$ be a finite set for $n=0,\dots,\ell$ such that 
    $C_n$ contains all critical points of $g_{n}$ in $V_n$ and 
    such that $g_{n}(C_n)\subset C_{n+1}$ for $n<\ell$. Then for every $\eta>0$ there exists $\delta>0$ with the following property.

 Suppose that, for every $n=0,\dots,\ell$, $f_n$ is a holomorphic function 
 on $U_n$ with $\lvert f_n - g_n\rvert <\delta$, and
   that furthermore $f_n|_{C_{n}}=g_n|_{C_{n}}$, with every critical point $c\in C_{n}$ of $g_n$ 
  also a critical point of $f_n$ of the same multiplicity. 

   Then for every $n=0,\dots,\ell$, there exists a
   conformal homeomorphism $\theta_n\colon V_n \to \theta_n(V_n)$ with
   $\theta_n(V_n)\subset U_n$ such that  
     \begin{enumerate}[(a)]
       \item $f_{n}\circ \theta_{n}= \theta_{n+1} \circ g_{n}$ on $V_{n}$ for  
          $n=0,\dots,\ell-1$,\label{item:thetanrelation}
        \item $f_{\ell}\circ \theta_{\ell} = g_{\ell}$ on $V_{\ell}$, 
       \item $\theta_n=\id$ on $C_n$ for $n=0,\dots,\ell$, and 
       \item $|\theta_n(z)-z|<\eta$ for all $z \in V_n$.\label{item:thetanclose}
       \end{enumerate}
\end{lem}
\begin{proof}
 We prove the claim by induction over $\ell$. For $\ell=0$, it is a consequence of Lemma~\ref{lem:theta-map}, 
  applied to $U=U_0$, $\phi=g_0$ and $V=V_0$. 

 Now suppose that the lemma is known for $\ell-1$, and let us prove it for $\ell$. Let $U\subset U_0$ be open with 
 $\overline{V_0}\subset U$ and $g_0(\overline{U})\subset V_1$. Apply Lemma~\ref{lem:theta-map} to 
 $\phi=g_0|_U$, with $V=V_0$ and  $\eps=\eta$. 
 
  We obtain a $\delta_0>0$ with the following property. Let 
  $f$ be holomorphic on $U$ 
  with $\lvert f - g_0\rvert < \delta_0$, and such that $f=g_0$ on $C_0$, 
  with the same multiplicity at critical points. Then there exists a conformal isomorphism $\theta$ on $V_0$ such that
  $f\circ \theta = g_0$, such that $\lvert \theta - \id\rvert < \eta$ on $V_0$
  and such that 
  $\phi = \id$ on $C_0$. 

 Now choose  $\delta_1>0$ so small that the following
 holds. If $f_0$ is holomorphic on $U_0$ with 
 $\lvert f_0 - g_0\rvert < \delta_1$ on $U_0$ and $\theta\colon V_1\to\C$ is conformal with $\lvert \theta - \id\rvert < \delta_1$
 on $V_1$ and $\theta|_{C_1}=\id$,
 then $\overline{f_0(U)}\subset \theta(V_1)$ and $\lvert \theta^{-1}\circ f_0 - g_0\rvert < \delta_0$ on $U$. Such a number $\delta_1$ exists because 
  $\theta^{-1}\circ f_0\to g_0$ uniformly on $U$ as $f_0\to g_0$ uniformly on
  $U_0$ and $\theta\to\id$ uniformly on $V_1$. 

 Now apply the inductive hypothesis to the maps $g_1,\dots,g_{\ell}$, 
   with $\tilde{\eta}\defeq \min(\eta,\delta_1)$. We obtain a number $\tilde{\delta}$ and set $\delta\defeq \min(\tilde{\delta},\delta_1)$.

  Now let the maps $(f_n)_{n=0}^{\ell}$ be as in the statement of the lemma (see Figure~\ref{fig:row}).
   We obtain $\theta_1,\dots,\theta_{\ell}$ with the desired properties from the
   inductive hypothesis. Moreover, again by the inductive hypothesis, 
   $\lvert \theta_1 - \id\rvert < \delta_1$ on $V_0$.
   Hence, by choice of $\delta_1$, 
   the map $f \defeq \theta_1^{-1}\circ f_0$ satisfies $\lvert f - g_0\rvert <
   \delta_0$ and $f=g_0$ on $C_0$. By choice of $\delta_0$,
   there is a map $\theta_0$ such that 
   $\theta_0 = \id$ on $C_0$, and such that $\lvert \theta_0(z)-z\rvert < \eta$  and
   $f_0\circ \theta_0 = \theta_1\circ f\circ\theta_0 = \theta_1\circ g_0$ on $V_0$.
\end{proof}

Now we show how we can extend a
 conformal isomorphism $\theta$ that is
 close to the identity to a diffeomorphism of the whole complex plane. 
 \begin{prop}\label{prop:diffeo}
  Suppose that $X\subset \R^2$ is compact and that $U\supset X$ is open. 
   Let $m\geq 1$, and let $\eps>0$.
   Then there is $\delta>0$ with the following property. 
   
  Suppose that $\theta\colon U\to\R^2$ is a diffeomorphism onto its image with $\|\theta - \operatorname{id}\|_{\mathcal{C}^m}<\delta$. Then 
    $\theta|_X$ extends to a $\mathcal{C}^{\infty}$ diffeomorphism $\tilde{\theta}\colon \R^2\to\R^2$ with 
    $\|\tilde{\theta} - \id\|_{\mathcal{C}^m} <\varepsilon$ and $\tilde{\theta}=\id$ on $\R^2\setminus U$.
 \end{prop}  
 \begin{proof}
   Let $V\Subset U$ be open with $X\subset V$. 
   Let $\psi\colon \R^2\to[0,1]$ be a smooth function such that
    $\psi=0$ on $\R^2\setminus V$ and $\psi=1$ on $X$. Let $z=(x,y)$ and define 
       \[ \tilde{\theta}(z) \defeq z + \psi(z)\cdot (\theta(z)-z). \] 
    Then, it follows by the definition of $\tilde{\theta}$ that  $\tilde{\theta}$ is $\mathcal{C}^{\infty}$ on $\R^2$, and also $\tilde{\theta}=\operatorname{id}$  
      on $\R^2\setminus V$ and $\tilde{\theta}=\theta$ on $X$. 

  We have $\| \psi\cdot (\theta - \id)\|_{\mathcal{C}^m}\to 0$ as $\| \theta - \id\|_{\mathcal{C}^m}\to 0$,
   since multiplication is continuous in the $\mathcal{C}^m$-topology. 
   %%% reference - e.g. 
   So we can choose $\delta>0$ such that, if
   $\|\theta - \id\|_{\mathcal{C}^m} < \delta$, then 
  \[ \|\tilde{\theta} - \id\|_{\mathcal{C}^m} <\min\Bigl(\frac{1}{2},\varepsilon\Bigr). \]
   The fact that $\|\Deriv \tilde{\theta} - \Deriv \id \| < 1$ implies that
       $\tilde{\theta}$ is a diffeomorphism, as required.
    \end{proof} 

\begin{cor}\label{cor:Cmextension}
 Let $m\geq 0$, let $\Theta\colon \C\to\C$ be
   a $\calC^{m}$ diffeomorphism and let $\eps>0$.

 Let $X\subset\C$ be compact and let
   $V\subset\C$ be an open neighbourhood of 
   $\Theta(X)$. Then there is $\eta>0$ with the following
  property. 

  Suppose that $\theta\colon V\to \theta(V)$ is
  a conformal isomorphism with 
  $\lvert \theta - \id\rvert < \eta$ on $V$. 
  Then there exists a $\calC^{\infty}$ diffeomorphism $\tilde{\theta}\colon\C\to\C$
  such that
  \begin{enumerate}[(a)]
   \item $\tilde{\theta}=\theta$ on
     a neighbourhood of $\Theta(X)$;\label{item:extendstheta}
   \item $\theta=\id$ outside $V$;\label{item:identityoutside} 
   \item $\| \tilde{\theta}\circ \Theta - \Theta\|_{\calC^m} < \eps$.\label{item:closetoTheta}
  \end{enumerate}
\end{cor}
\begin{proof}
 Let $\tilde{V}\subset V$ be a neighbourhood 
  of $\Theta(X)$ that is compactly contained 
  in $V$. By Cauchy's integral formula, we
  have $\|\theta - \id\|_{\calC^m}\to 0$ on
  $\tilde{V}$ as 
  $\lvert \theta - \id\|\to 0$ on $V$. 

 Hence the existence of an extension 
   $\tilde{\theta}$ with
   properties~\ref{item:extendstheta} and~\ref{item:identityoutside} follows
   from Proposition~\ref{prop:diffeo} for
   sufficiently small $\eta$,
   and moreover $\|\tilde{\theta}-\id\|_{\calC^m}\to 0$ as $\eta\to 0$. 

  Since composition is continuous in the
   $\calC^m$-topology, it follows that
   also $\| \tilde{\theta}\circ \Theta - \Theta\|_{\calC^m}\to 0$ as $\eta\to 0$;
   hence~\ref{item:closetoTheta} holds
   for sufficiently small $\eta$. 
\end{proof}

\section{Conjugacies for non-autonomous sequences}

\label{sec:conjugacies-nonautseq}

 We will now establish the key ingredient in the proof of Theorems~\ref{thm:Jordan-continua merged} and Theorem~\ref{thm:Jordan-continua merged-oscillating}.
 Similarly as in the papers~\cite{bocthaler21,mrw1,mrw2}, the idea is to construct the desired  meromorphic function as a limit of functions that are 
 defined by approximation, and exhibit the desired behaviour on neighbourhoods of more and more of the compact sets in
 question. The key point is that, if in every step we choose the next function sufficiently close to $\phi$ on the set $X_n$, and
 close to the previously constructed function on $X_0,\dots,X_{n-1}$, then the resulting limit function will be conjugate to $\phi$ in
 the manner that we desired.

 This idea is applicable in a wide variety of contexts, beyond the context of our article.
 Therefore we will phrase it in a more general, but also more
  abstract setting, for the purpose of future applications. Let us begin by defining the class
  of systems that we will be approximating (of which the setting in Theorem~\ref{thm:Jordan-continua merged}
  is a special case); recall that we denote the set of critical points of a function $\phi$ by $\CP(\phi)$.

\begin{dfn}[Holomorphic non-autonomous systems]
  A \emph{marked holomorphic non-auto\-nomous system on compact sets} is a 
   triple $\Phi = \bigl((X_n)_{n=0}^{\infty},(\phi_n)_{n=0}^{\infty},(C_n)_{n=0}^{\infty}\bigr)$,
   where each $X_n$ is a compact subset of $\C$, each $\phi_n$ is a holomorphic
     function defined on an open neighbourhood of $X_n$
     such that $\phi_n(X_n)\subset X_{n+1}$, and each $C_n$ is a finite set
     with $\CP(\phi_n)\cap X_n \subset C_n$ and $\phi_n(C_n)\subset C_{n+1}$.   
\end{dfn}
\begin{rmk}
    For brevity, in the following we will write ``non-autonomous system'' instead of ``marked holomorphic non-autonomous
  system on compact sets'', since this is the only type of system that we consider in this article. 
  \end{rmk}

  We now wish to consider recursive constructions of the following form: 
   We choose (in some way) a function $f_{0,1}$ that is close to $\phi_0$ on a neighbourhood of $X_0$.
   Then we choose a function $f_{0,2}$ that is close to $f_{0,1}$ (with the same domain of definition),
   and a function $f_{1,2}$ that is close to $\phi_1$ on a neighbourhood of $X_1$. Inductively, in step $j$
   we choose functions $f_{n,j}$ close to $f_{n,j-1}$ for $n< j-1$, and $f_{j-1,j}$ close to $\phi_{j-1}$. 
   The goal is to show that, if the approximation is chosen sufficiently close at every step (and 
   respects critical points in an appropriate manner), then
   for each $n$, the functions $f_{n,j}$ converge to a function $f_n$, and furthermore these limiting 
   maps form a non-autonomous system (on a suitable sequence of compact sets) that is conjugate
   to the original one. 

 To make this type of construction more precise, it is useful to introduce the following
  notation for the ``triangles'' of functions $(f_{n,j})_{0\leq n < j}$ obtained in a construction
  as above. All the objects 
  introduced in the following depend on a given non-autonomous system $\Phi$; for simplicity of
  notation we will suppress this dependence.

\begin{dfn}[Triangles]\label{defn:triangles}
  Let $\Phi = \bigl((X_n)_{n=0}^{\infty},(\phi_n)_{n=0}^{\infty},(C_n)_{n=0}^{\infty}\bigr)$ 
  be an non-autonomous system. 
  For $n\geq 0$, we denote by $\U_n(\Phi)$ the set of all open neighbourhoods
  of $X_n$ on which $\phi_n$ is defined.

  A \emph{finite triangle} $\mathcal{T}$ (of level $k\geq 0$) is a collection $\mathcal{T} = (f_{n,j})_{0\leq n < j \leq k}$ 
   of holomorphic functions such that the following conditions hold.
   \begin{enumerate}[(a)]
     \item Let $0\leq n < k$. Then all $f_{n,j}$, where $j=n+1,\dots,k$, are 
       defined and holomorphic on the same domain $U_{n}  = U_{n}(\mathcal{T})\in \U_n(\Phi)$. 
     \item For every $z\in C_{n}$ and
     $n< j \leq k$, we have $f_{n,j}(z) = \phi_{n}(z)$ and the local degree of $f_{n,j}$ near
     $z$ is the same as that of $\phi_{n}$.
   \end{enumerate}
  The set of finite triangles of level $k$ is denoted by $\Delta_k(\Phi)$; observe that $\Delta_0(\Phi)$ only contains
   one element (the triangle whose index set is empty). 
   
  An \emph{infinite triangle} is a collection $\mathcal{T} = (f_{n,j})_{0\leq n < j < \infty}$ 
    such that $\mathcal{T}_k\defeq (f_{n,j})_{0\leq n < j \leq k}\in \Delta_k(\Phi)$ 
    for all $k\geq 0$. Note that $f_{n,j}$ is defined on a domain $U_{n}(\mathcal{T})\in \U_n(\Phi)$ that depends 
      on $n$ but is independent of $j$.
\end{dfn}

\begin{figure}[t]
\[
\xymatrix{
X_{0} \ar[d]_{\theta_{0,0}} \ar[dr]^{\varphi_0} \\
X_{0,1} \ar[d]_{\theta_{0,1}} \ar[r]^{f_{0,1}} & X_{1} \ar[d]_{\theta_{1,1}} \ar[dr]^{\varphi_1} \\
X_{0,2} \ar[d]_{\theta_{0,2}} \ar[r]^{f_{0,2}} & X_{1,2} \ar[d]_{\theta_{1,2}} \ar[r]^{f_{1,2}} & X_{2} \ar[dr]^{\varphi_2} \ar[d]_{\theta_{2,2}}\\
X_{0,3}  \ar[r]^{f_{0,3}}\ar[d]_{\theta_{0,3}} & X_{1,3} \ar[r]^{f_{1,3}}\ar[d]_{\theta_{1,3}}& X_{2,3}\ar[r]^{f_{2,3}} \ar[d]_{\theta_{2,3}}& X_{3} \ar[d]_{\theta_{3,3}} \ar[dr]^{\varphi_3} \\
\vdots & \vdots & \vdots & \vdots & \ddots
}
\] 
\caption{Diagram for the proof of Theorem~\ref{thm:triangle1}. 
For $0\leq n\leq j$, the restriction of $\Theta_{n,j}$ to $X_{n,n}=X_n$, is the composition $\theta_{n,j-1}\circ\dots\circ\theta_{n,n}$, which maps $X_n$ to $X_{n,j}$. 
In particular, $\Theta_{j,j}=\id$ and $\Theta_{j,j+1}=\theta_{j,j}$.}
\label{fig:mainabstract}
\end{figure}
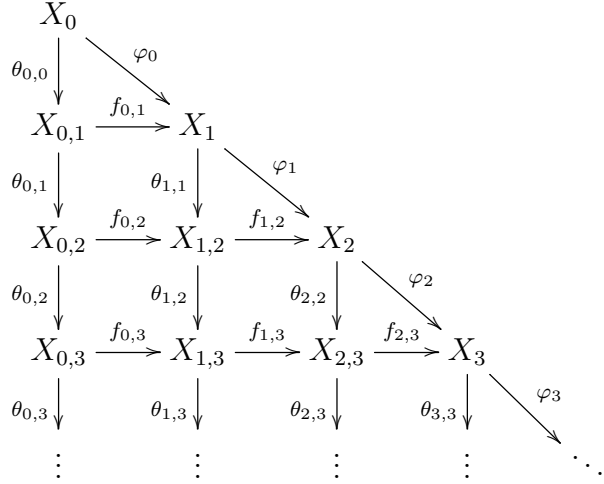   

\pagebreak

With this notation, the main result of this section is the following. 
\begin{thm}[Conjugacy for close systems]\label{thm:triangle1}
    Let $\Phi = \bigl((X_n)_{n=0}^{\infty},(\phi_n)_{n=0}^{\infty},(C_n)_{n=0}^{\infty}\bigr)$ be a non-autonomous system.
      Let $(\varepsilon_n)_{n=0}^\infty$ be a sequence of positive real numbers and let
 $(M_n)_{n=0}^{\infty}$ be a sequence of non-negative integers. 

     Then there is a sequence $(\delta_j)_{j=1}^{\infty}$ of functions
         \begin{equation}\label{eqn:deltaj} \delta_j\colon \Delta_{j-1}(\Phi) \times \U_{j-1}(\Phi)\to (0,\infty)\end{equation}
       with the following property. 

  Suppose that $\mathcal{T} = (f_{n,j})_{0\leq n < j  <\infty}$ is an infinite triangle such that,
    for all $j\geq 1$,
    \begin{align}    \label{eqn:phclose} \lvert f_{j-1,j} - \phi_{j-1}\rvert &< \delta_j(\mathcal{T}_{j-1},U_{j-1}(\mathcal{T}))
    \quad\text{and} \\
    \label{eqn:fclose} \lvert f_{n,j}  - f_{n,j-1} \rvert &< \delta_{j}(\mathcal{T}_{j-1},{U}_{j-1}(\mathcal{T})) \quad\text{for $0\leq n<j-1$.}\end{align}
   (Here  $\T_{j-1}\in\Delta_{j-1}(\Phi)$ is the finite triangle of level $j-1$ 
    determined by
   $\T$.)

  Then for  every $n\geq 0$, the sequence $(f_{n,j})_{j=n+1}^{\infty}$ converges uniformly to a holomorphic map
    $f_{n}\colon U_{n}(\T)\to\C$ and there exists a sequence of diffeomorphisms 
     $\Theta_{n}\colon\C\to\C$ with the following properties.
\begin{enumerate}[(a)]
\item $f_{n} \circ \Theta_n = \Theta_{n+1} \circ \phi_{n}$ on $X_n$ for all $n\geq 0$;
\item\label{item:Thetaderivative} 
    $\|\Theta_n(z)-\id\|_{\calC^{M_n}} \leq \eps_n$ on $U_n(\T)$ for all $n\geq 0$;
\item $\Theta_n=\id$ outside $U_n(\T)$ and on $C_n$ for all $n\geq 0$;
\item $\Theta_n$ satisfies the Cauchy--Riemann differential equations on $X_n$; in particular, it is conformal on $\operatorname{int}(X_n)$.
\end{enumerate}    
\end{thm}
\begin{proof}
 We construct the 
   functions $(\delta_j)_{j=1}^\infty$  recursively, ensuring
   inductively that the following hold. 
   Set $m_{n,j}\defeq \max(M_n,j)$ for
   $0\leq n \leq j$. 
   Suppose that 
   $\T$ is a triangle as in the statement of the theorem. Then there exists a collection
   $(\Theta_{n,j})_{0\leq n \leq j < \infty}$ of $\calC^{m_{n,j}}$ diffeomorphisms
   $\Theta_{n,j}\colon \C\to\C$
   with the folllowing properties. 
   \begin{enumerate}[1.]
    \item $\Theta_{j,j}=\id$ for all $j\geq 0$.
    \item Each $\Theta_{n,j}$ is conformal in 
      a neighbourhood of $X_n$.
    \item $\Theta_{n,j}=\id$ outside $U_n$ and on $C_n$. 
    \item $\Theta_{n+1,j}\circ \phi_j = f_{n,j}\circ \Theta_{n,j}$ on an open neighbourhood of $X_n$ when 
      $0\leq n < j$.\label{item:Thetarelation}
     \item $\| \Theta_{n,j}-\Theta_{n,j+1}\|_{\calC^{m_{n,j+1}}}<\eps_n/2^{j+1}$ when $0\leq n\leq j$.
     \item $\delta_j(\T_{j-1},U_{j-1})\leq 2^{-j}$ for all $j$. 
   \end{enumerate}
   Moreover, for fixed $j\geq 0$ the maps $(\Theta_{n,j})_{n=0}^j$ depend 
   on $\T_j$, but not on other entries of
   the triangle $\T$. 

 We now show how to define 
  $\delta_j$ with the above properties,
  by using Lemma~\ref{lem:theta-n} and
  Corollary~\ref{cor:Cmextension} inductively. 

 Let $j\geq 1$ and 
   suppose that a finite triangle 
   $\T_{j-1}$ of level $j-1$ is given such
   that the maps $(\Theta_{n,j-1})_{n=0}^{j-1}$ 
   have already been defined with the above
   properties. Also let $U_{j-1}\in\U_{j-1}(\Phi)$. We must show how to 
   define $\delta_j = \delta_j(\T_{j-1}, U_{j-1})$ in such a way that 
   the maps $(\Theta_{n,j})_{n=0}^{j}$ exist
   for any triangle $\T_j$ extending $\T_{j-1}$ 
   and satisfying~\eqref{eqn:phclose} and~\eqref{eqn:fclose}. 

  Set $X_{n,j-1}\defeq \Theta_{n,j-1}(X_n)$ 
  for
  $0\leq n \leq j-1$. Define $U_n\defeq U_n(\T_{j-1})$ for $n<j-1$ and
  let $(g_n)_{n=0}^j$ be the sequence of holomorphic functions defined by
   $g_n\defeq f_{n,j}$ for $n<j-1$ and $g_{j-1}\defeq \phi_{j-1}|_{U_{j-1}}$. 
   Now choose open sets
   $V_n = V_{n,j-1}$ with $X_{n,j-1}\subset V_{n}\Subset U_n$
   and $g_n(\overline{V_n})\subset V_{n+1}$ for $n=0,\dots,j-2$. To see that this is possible,
   let $V_{j-1}\Subset U_{j-1}$  be any open neighbourhood of $X_{j-1,j-1}=X_{j-1}$, and recursively choose
   $V_{n-1}$ such that $g_{n-1}(\overline{V_{n-1}})\subset V_n$; this is possible since 
   $g_{n-1}(X_{n-1,j-1}) \subset X_{n,j-1}$. We may also assume that $V_n$ is chosen
   sufficiently small to ensure that all critical points of $g_n|_{V_n}$ are in $X_{n,j-1}$, and hence in 
   $C_n$. 

 For $n=0,\dots,j-1$, apply Corollary~\ref{cor:Cmextension} to $m=m_{n,j}$, $\Theta=\Theta_{n,j-1}$, 
  $\eps=\eps_n/2^{j}$, $X=X_n$ and $V=V_{n,j-1}$. We obtain a number $\eta_{n,j}$ with the
  following property: If $\theta$ is a conformal isomorphism defined on $V_{n,j-1}$ and
  satisfying $\lvert \theta - \id\rvert < \eta_{n,j}$, then there exists a $\calC^{\infty}$ diffeomorphism
  $\tilde{\theta}\colon\C\to\C$ that agrees with $\theta$ on a neighbourhood of $X_{n,j-1}$, is the identity 
  outside $V_{n,j-1}$ (and hence outside $U_n$), and $\theta\circ \Theta_{n,j-1}$ 
  is $\eps$-close to
  $\Theta_{n,j-1}$ in the $\calC^{m_{n,j}}$-norm. 

 Set $\eta_j\defeq \min_{0\leq n \leq j-1} \eta_{n,j}$ and apply Lemma~\ref{lem:theta-n}, where
 $\ell=j-1$ and $\eta=\eta_j$. 
    The domains $(U_n)_{n=0}^{\ell}$ and $(V_n)_{n=0}^{\ell}$, the maps $(g_n)_{n=0}^{\ell}$ 
    and the sets $C_n$ are as defined above, and satisfy the hypotheses of the lemma by construction.
    We obtain a number
    $\delta$, and define $\delta_j(\T_{j-1},U_{j-1})\defeq \min(\delta,2^{-j})$
    are as above, and domain, with $\ell=j-1$ and 
    $\eps=\tilde{\eps}_j$, to obtain a number $\delta$. Define 
    $\delta_j(\T_{j-1},U_{j-1}(\T))\defeq \min(\delta,2^{-j})$. 
    
  Now suppose that $\T_{j}$ is a finite triangle extending $\T_{j-1}$ and 
    satisfying~\eqref{eqn:phclose} and~\eqref{eqn:fclose}. Then
    the maps $f_{n,j}$ satisfy the conditions in Lemma~\ref{lem:theta-n}, and hence
    there exist conformal isomorphisms $\theta_{n,j}$ defined on $V_{n,j}$ satisfying properties~\ref{item:thetanrelation}--\ref{item:thetanclose} of that lemma. By choice of $\eta$, for each $\theta_{n,j}$ there is 
    a $\calC^{\infty}$ diffeomorphism $\tilde{\theta}_{n,j}$ as above, and the map
    $\Theta_{n,j+1}\defeq \tilde{\theta}_{n,j}\circ \Theta_{n,j}$ has the desired properties.

  This completes the recursive definition of the functions $\delta_n$. If $\T$ is an infinite
  triangle as in the statement of the theorem, let the maps $(\Theta_{n,j})_{0\leq n\leq j < \infty}$ be
  as defined above. Then, for every fixed $n$ and every $m\geq 0$, the sequence $(\Theta_{n,j})_{j\geq n}$ forms 
  a Cauchy sequence in the $\calC^m$ topology. Hence 
  it converges to a $\calC^{\infty}$ diffeomorphisms 
  $\Theta_n$ with the desired properties.
\end{proof}

Recall that we wish to use Theorem~\ref{thm:triangle1} to construct wandering compacta of entire 
 functions with prescribed dynamics. To ensure that the boundaries of these compacta are contained in
 the Julia set, we need the following additional information.

\begin{thm}\label{thm:triangleaddendum}
    In Theorem~\ref{thm:triangle1}, assume additionally that we are given
    a sequence $(K_n)_{n=0}^{\infty}$ of compact sets $K_n\subset X_n$ such that 
    $\phi_n(K_n)\subset K_{n+1}$ for all $n\geq 0$.
     Then in addition to the functions $\delta_j$,
     there is a sequence of functions
     \[ \hat{\delta}_j \colon \Delta_{j}(\Phi) \to (0,\infty), \]
     where $j\geq 1$, with the following property. 
     
     Let $\T$ be a triangle satisfying the assumptions of
     Theorem~\ref{thm:triangle1} and let $(f_n)_{n=0}^{\infty}$ be the sequence of limit functions
     whose existence is guaranteed by the theorem. 
     For every $j\geq 0$, choose a compact
     set $Q_j$ whose $\hat{\delta}_j(\T_j)$-neighbourhood contains $K_j$. 
     Then, for every
     $n\geq 0$  and every $j_0 > n$,
     \[ \Theta_n(K_n) 
     \subset \overline{\bigcup_{j=j_0}^{\infty} (f_{j-1}\circ\dots\circ f_n)^{-1}(Q_j)}.\]
\end{thm}
\begin{proof}
 We may assume without loss of generality that $\eps_n\to 0$ as $n\to\infty$, 
 where $(\varepsilon_n)_{n=0}^{\infty}$ is the sequence from
  Theorem~\ref{thm:triangle1}. 
  We continue with the notation from the proof of that theorem.

  Fix $j\geq 1$. 
   For $n=0,\dots,j-1$, apply Lemma~\ref{lem:preimages-points} to the composition $F_{n,j}\defeq f_{j-1,j}\circ \dots \circ f_{n,j}\colon V_{n,j}\to \C$, the set $K_{n,j}\defeq \Theta_{n,j}(K_n)$
   and 
   $\eps\defeq 1/j$. We obtain
  a number $\hat{\delta}_{n,j}$ and define 
    \[ \hat{\delta}_j\defeq \min_{0\leq n \leq j-1} \hat{\delta}_{n,j}. \]
  Recall that the sets $V_{n,j}$ depend only on $\T_j$, and hence $\hat{\delta}_j$ is indeed a function
  on $\Delta_j(\Phi)$. 
  
  We also require that for $j'>j$, the numbers $\delta_{j'}$ in the proof of Theorem~\ref{thm:triangle1} are
   chosen sufficiently small to ensure that 
      \begin{equation}\label{eqn:compositionclose} \| f_{j-1,j'} \circ \dots \circ f_{n,j'} - F_{n,j}\| < \hat{\delta}_j \end{equation}
      on $V_{n,j}$ for all $n\leq j-1$ and all $j'> j$ when the 
      triangle $\T$ satisfies~\eqref{eqn:phclose} and~\eqref{eqn:fclose}. This is possible by Lemma \ref{lem:approx-iterates}.

  Now suppose that $(Q_j)_{j=j_0}^{\infty}$ is a set as in the statement of the theorem, and let
   $n\geq 0$.  By~\eqref{eqn:compositionclose}, the map 
   \[ \tilde{F}_{n,j}\defeq f_{j-1}\circ \dots \circ f_n \] 
   satisfies $\| \tilde{F}_{n,j} - F_{n,j}\| \leq \hat{\delta}_j$ on $V_{n,j}$ for $n\leq j-1$. Therefore, by Lemma~\ref{lem:preimages-points},
   the $1/j$-neighbourhood of $P_j\defeq \tilde{F}_{n,j}^{-1}(Q_j)$ contains 
   $K_{n,j}$. Since $K_{n,j}\to \Theta_n(K_n)$ 
   as $j\to\infty$, the claim follows.
\end{proof}

In applications of Theorems~\ref{thm:triangle1} and~\ref{thm:triangleaddendum},
 the functions
 in the triangle $\T$ are typically constructed recursively. In some cases, one may not wish to specify all
 functions $\phi_j$, sets $X_j$ or the numbers $M_j$ and $\eps_j$ in advance of the whole construction, but
 rather to choose these, too, in a recursive manner. Therefore we note the following. 
\begin{prop}
 In Theorem~\ref{thm:triangle1}, for each $j\geq 1$ the function $\delta_j$ depends on  $(\phi_k)_{k=0}^{j-1}$, $(X_k)_{k=0}^{j-1}$, 
 $(C_k)_{k=0}^{j-1}$, $(\eps_k)_{k=0}^{j-1}$ and $(M_k)_{k=0}^{j-1}$, 
  but not on the values of these sequences for other values of $k$.
  In Theorem~\ref{thm:triangleaddendum}, the function $\hat{\delta}_j$ depends 
  additionally on $(K_k)_{k=0}^{j}$. 
\end{prop}
\begin{proof}
 The first claim follows easily by inspecting the proof. 

 For the second, note that $\hat{\delta}_j$ depends on the sets $V_{n,j}$ for
   $n\leq j-1$,  the maps $\Theta_{n,j}$ for $j\leq j-1$, the sets
   $X_{n,j}=\Theta_{n,j}(X_n)$ for $j\leq j-1$, and, of course, on $X_j$. Again, inspecting the
   proof of Theorem~\ref{thm:triangle1} shows that these objects are independent of the input data
   at indices above $j-1$. 
\end{proof}

\section{Proofs of Theorems \ref{thm:Jordan-continua merged}, \ref{thm:Jordan-continua merged-oscillating} and \ref{thm:entire}}
\label{sec:main-proof}

\begin{proof}[Proof of Theorem~\ref{thm:Jordan-continua merged}]
We construct the transcendental meromorphic function $f$ as the limit of a sequence of rational maps $(f_j)_{j=0}^\infty$ that are defined inductively using Runge's theorem (see Theorem~\ref{thm:runge}).  More precisely, for $j\geq 1$, the function $f_j$ approximates a function $g_j$, defined and meromorphic on a neighbourhood of a compact set $A_j\subseteq \C$, up to an error of at most $\varepsilon_j>0$. The function $g_j$ in turn is defined in terms of the previous function $f_{j-1}$ and $\varphi$.  The sets $(A_j)$ where we perform the approximation satisfy  $A_j\subseteq A_{j+1}$ for $j\geq 1$ and their union is $\C$. This means that, assuming that $f_j$ and $f_{j-1}$ are sufficiently close on $A_j$ (in particular, $f_j$ and $f_{j-1}$ have the same poles in $A_j$), in the limit we obtain a meromorphic function $f\colon\C\to\widehat{\C}$. Note that since the sets $(X_n)$ are bounded, $\theta(X_n)$ never contains any poles, but $f$ may have poles close to $\theta(X_n)$ for all $n \geq 0$; we treat the entire case at the end of the proof.

 For each $n\geq 0$, choose a bounded open neighbourhood $W_n$ of $X_n$ on which $\phi$ is defined. 
   By the assumption on the $X_n$ we can choose $W_n$ so that
   their closures  are pairwise disjoint and tend uniformly to infinity as $n\to\infty$. 
   We may also suppose that $W_n$ contains no critical points of $\phi$ outside
   $X$, and that $\CP(\phi)\cap X_n\subset C_n$ and $\phi(C_n)\subset C_{n+1}$ for all
   $n\geq 0$.

Let $D=D(c,\rho)$ be a disc whose closure is disjoint from the closures of the $W_n$ and define $f_0(z)\defeq c$. 
 Choose an increasing sequence $(\Lambda_n)_{n=0}^{\infty}$ of compact sets such that 
   $\overline{D}\subset \Lambda_0$, $W_j\subset \Lambda_n$ for $j<n$ and $\overline{W_j}\cap \Lambda_n=\emptyset$ for $j\geq n$, and $\bigcup_{n=0}^\infty \Lambda_n = \C$ . (Note that such a sequence $(\Lambda_n)_{n=0}^{\infty}$ exists by choice of the $W_n$.)   

Let $\Phi = \bigl((X_n)_{n=0}^{\infty},(\phi_n)_{n=0}^{\infty}\bigr)$ be the non-autonomous system defined 
  by $\phi_n=\phi|_{W_n}$ for all $n$. Then $\Phi$ satisfies the hypotheses of Theorems~\ref{thm:triangle1} and~\ref{thm:triangleaddendum}, 
  where $K_n=\partial X_n$. Let  $\delta_n$ and $\hat{\delta}_n$ be the functions whose existence
  is guaranteed by these theorems. (Here the sequence $\eps_n$ is as given in the statement of the theorem,
  and we may take $M_n=0$ for all $n\geq 0$.)

We construct the rational maps $(f_j)_{j=1}^{\infty}$ recursively together with open sets $(U_n)_{n=0}^{\infty}$, in such a way that the 
 restrictions $f_{n,j}\defeq f_j|_{U_n}$, for $n<j$, form an infinite triangle $\T$ in the sense of
 Definition~\ref{defn:triangles}
  and satisfying conditions~\eqref{eqn:phclose} and~\ref{eqn:compositionclose} from Theorem~\ref{thm:triangle1}. Furthermore, $|f_{j}-f_{j-1}| \leq \rho/2^{j+1}$
 on $\Lambda_{j-1}$. In particular, every $f_j$ maps $D(c,\rho)$ into $D(c,\rho/2)$. 

Let $j\geq 1$. Assume that $f_k$ is defined for $k< j$, that $U_n$ is defined for $n<j-1$,
 and that the maps $f_{n,k} = f_k|_{U_n}$, where $0\leq n < k \leq j-1$, 
 define a finite triangle $\T_{j-1}$. Set
 $\hat{\delta}\defeq \hat{\delta}_{j-1}(\T_{j-1})$. Choose 
 a compact set $Q_{j-1}\subset W_{j-1}\setminus X_{j-1}$ such that the $\hat{\delta}$-neighbourhood of 
 $Q_{j-1}$ contains $\partial X_{j-1}$. Let $R_{j-1}$ be a compact neighbourhood of $Q_{j-1}$ contained
 in $W_{j-1}$ and disjoint from $X_{j-1}$. Finally, choose $U_{j-1}\subset W_{j-1}$ to be an open
 neighbourhood of $X_{j-1}$ such that $\overline{U_{j-1}} \cap R_{j-1} = \emptyset$. 

 Define 
   \begin{equation}\label{eqn:Ajescaping} A_j \defeq \Lambda_{j-1} \cup \overline{U_{j-1}} \cup R_{j-1}\end{equation}
   and
     \[ g_j(z) \defeq \begin{cases} f_{j-1}(z) &\text{if }z\in \Lambda_{j-1} \\
                                    \phi_{j-1}(z) &\text{if }z\in \overline{U_{j-1}} \\
                                    c &\text{if }z\in R_{j-1}. \end{cases}\]
    Observe that the three sets in the definition of $A_j$ are compact and pairwise disjoint, and that
     $g_j$ is meromorphic in a neighbourhood of $A_j$. 

 Set 
   \[ \delta \defeq \min\bigl(2^{-(j+1)}\rho , \delta_n(\T_{j-1},U_{j-1})\bigr). \]
  Let $(C_n)_{n=0}^{\infty}$ be the sets defined in Definition~\ref{defn:triangles}. By Theorem \ref{thm:runge} we can choose
  $f_j$ to be a rational function that is $\delta$-close to $g_j$ on $A_j$, and that
  takes the same values and has the same local degree on the points of $\bigcup_{k=0}^{j-1} C_k$ as $g_j$. In particular, $\T_j$ is a triangle of level $j$, as required by the inductive hypothesis.

This completes the recursive construction. By construction, the infinite
 triangle $(f_{n,j})_{0\leq n < j<\infty}$ satisfies the hypotheses of Theorems~\ref{thm:triangle1}
 and~\ref{thm:triangleaddendum}.
 Moreover, the functions $f_j$ converge to a global meromorphic function $f$ as $j\to\infty$. 

 Let $(\Theta_n)$ be the functions from Theorem~\ref{thm:triangle1}, and define
   $\theta\colon\C\to\C$ to agree with $\Theta_n$ on each $U_n$, and to be the identity outside
   the union of the $U_n$. Then~\ref{item:functionalrelation}, \ref{item:conjugacyclose} and~\ref{item:conjugacyconformal} follow from the conclusion of Theorem~\ref{thm:triangle1}. Furthermore,
   $f(Q_n)\subset D$ for all $n\geq 0$, and $f(D)\subset D$. So all points in $Q_n$ have bounded orbits. 
   The union of the backward orbits of all sets $Q_n$ accumulate on $\partial \theta(X)$ by
   Theorem~\ref{thm:triangleaddendum}. On the other hand, $\theta(X)\subset I(f)$. We conclude that
   $\partial \theta(X)\subset J(f)$, as claimed in~\ref{item:boundaryJulia}. This completes the
   proof in the meromorphic case. 
\end{proof}

\begin{rmk}
 Note that the assumption that $\phi(\partial X_n)\subset \partial X_{n+1}$ was only
  used to ensure that $K_n\defeq \partial X_n$ satisfy the hypotheses of
  Theorem~\ref{thm:triangleaddendum}. Without this assumption, we can instead 
  set 
   \[ K_n\defeq \bigl\{z\in \partial X_n\colon \phi_{\ell}(\phi_{\ell-1}(\cdots(\phi_n(z))\cdots))\in \partial X_{\ell}
       \text{ for all $\ell>n$}\bigr\}\]
 and the proof goes through verbatim. This proves the existence of
 a meromorphic function $f$ satisfying conclusions~\ref{item:functionalrelation},~\ref{item:conjugacyclose} and~\ref{item:conjugacyconformal} of Theorem~\ref{thm:main-escaping}, while~\ref{item:boundaryJulia} should be replaced by 
 \begin{enumerate}
   \item[(c')] $\Theta(K_n)\subset J(f)$ for all $n\geq 0$.
  \end{enumerate}
\end{rmk}

Similarly, we can apply Theorem~\ref{thm:triangle1} to prove Theorem~\ref{thm:Jordan-continua merged-oscillating}.

\begin{proof}[Proof of Theorem~\ref{thm:Jordan-continua merged-oscillating}]
The proof is similar to the proof of Theorem~\ref{thm:main-escaping}. In particular,
the overall outline of the construction is the same: we recursively construct a sequence 
$(f_j)_{j=0}^{\infty}$ of rational maps, each of which approximates a function
$g_j$ that is defined in terms of the previous steps of the construction. The non-autonomous system $\Phi$ is defined in the same way as in the proof of Theorem~\ref{thm:main-escaping}, but now the
triangle $\T$ is given by the maps 
$f_{n,j}\defeq f_j^{k_{n}}|_{U_n}$. (As before, $U_n$ is a neighbourhood of $X_n$ that
is chosen inductively as part of the construction.) The sequence $(k_n)_{n=0}^\infty$ is given by 
\[ k_n \defeq n+2,\quad \textup{for } n\geq 0.\]

Choose the open sets $(W_n)_{n=0}^{\infty}$, the disc $D$ and  the map 
 $f_0$ as in the proof of Theorem~\ref{thm:main-escaping}. Also let
 $(D_n)_{n=0}^\infty$ be a collection of discs whose closures are disjoint from $\overline{D}$ and the closures of the sets $W_n$, for $n\geq 0$, and such that $\inf_{z\in D_n} |z|\to\infty$ as $n\to\infty$. Also, for each $n\geq 0$, 
choose two smaller subdiscs $\Omega_n,\Sigma_n$ whose closures are 
disjoint from each other and contained in $D_n$. Let $\omega_n$ and $\sigma_n$
  be affine maps such that   
 \[ \overline{W_{n+1}}\subset \omega_n(\Omega_n)\quad\text{and}\quad  
    \overline{D_{n+1}}\subset \sigma_n(\Sigma_n). \]
 Furthermore, choose $\alpha_n>0$ be so small that the following holds. 
 Suppose that $f$ is a holomorphic function on $\Omega_n\cup \Sigma_n$ such that
 $\lvert f - \omega_n\rvert \leq \alpha_n$ on $\Omega_n$ and
 $\lvert f - \sigma_n\rvert \leq \alpha_n$ on $\Sigma_n$.
 Then $\Omega_n$ contains a simply connected domain that is 
  mapped conformally onto $W_{n+1}$ by $f$, and similarly for $\Sigma_n$ and
 $D_{n+1}$. This is possible by Lemma~\ref{lem:univ-approx}.

In addition to the conditions in the proof of Theorem~\ref{thm:main-escaping},
  the compact sets $(\Lambda_n)_{n=0}^\infty$ are chosen such that $D_j\subset \Lambda_n$ for
  $j<n$ and $\overline{D_j}\cap \Lambda_n=0$ for $j\geq n$. As before, the maps
  $(f_j)_{j=0}^\infty$ and domains $(U_n)_{n=0}^\infty$ are chosen such that the maps $f_{n,j}$ defined above
  form a triangle satisfying~\eqref{eqn:phclose}~and~\eqref{eqn:fclose},
  and such that $f_j$ maps $D$ into itself. Moreover, %$f_j$ will be $\alpha_k\cdot (1-2^{-j})$-close to $\omega_k$ and $\sigma_k$ for $k< j$.
  $|f_j- \omega_k|< \alpha_k\cdot (1-2^{-j})$ and $|f_j- \sigma_k|< \alpha_k\cdot (1-2^{-j})$ for $k< j$.

 We now describe the recursive construction of $f_{j}$ from $f_{j-1}$. The sets
   $Q_{j-1}$, $R_{j-1}$ and $U_{j-1}$ are defined as before, but
   we also require that $\phi_{j-1}(\overline{U_{j-1}}) \subset W_j$. The set $A_j$,
   on which the approximation is to take place, is
   \begin{equation}\label{eqn:Ajoscillating} A_j\defeq \Lambda_{j-1} \cup \overline{U_{j-1}} \cup R_{j-1}
   \cup\overline{\Omega_{j-1}}
        \cup \overline{\Sigma_{j-1}}.\end{equation}

   By the inductive hypothesis on $f_{j-1}$ and choice of $\alpha_0,\dots,\alpha_{j-1}$,
     there is a branch $\beta$ of $f_{j-1}^{-(j-1)}$ defined on $D_n$ and taking
     values in $D_0$. (Here $\beta=\id$ when $j=1$.) We define 
       \[ g_j(z) \defeq 
          \begin{cases}  f_{j-1}(z) & \text{if }z\in \Lambda_{j-1} \\
              \beta(\omega_{j-1}^{-1}(\phi_{j-1}(z))) 
                                &\text{if } z\in \overline{U_{j-1}} \\
              c &\text{if } z\in R_{j-1} \\
              \omega_{j-1}(z) &\text{if } z\in \overline{\Omega_{j-1}} \\
              \sigma_{j-1}(z) &\text{if } z\in \overline{\Sigma_{j-1}}. 
          \end{cases}\]

Observe that $g_j^{k_j} = \phi_{j-1}$ on $\overline{U_{j-1}}$, and that
      $g_j$ extends holomorphically to a neighbourhood of $A_j$. Hence, we may choose
      $\delta>0$ and apply Theorem~\ref{thm:runge} to choose a rational function $f_j$
      %that is $\delta$-close to $g_j$ 
      such that $|f_j-g_j|<\delta$ on $A_j$ and has the same values 
      and degree at all points of the finite set 
       \[ \bigcup_{m=0}^{j-1} \Bigl\{g_j^{\ell}(c)\colon c\in C_m, 0\leq \ell <
              \sum_{q=m}^{j-1} k_{q} \Bigr\}.\]
    The number $\delta$ should be chosen sufficiently small to ensure that 
       \[
             \delta \leq \min\Bigl(2^{-(j+1)}\rho , \delta_j(\T_{j-1},U_{j-1}),\min_{m<j} \frac{\alpha_m}{2^{j}}\Bigr), 
        \]
      and such that $\lvert f_j^{n-j} - \phi_{j-1}\rvert < \delta_j(\T_{j-1},U_{j-1}))$
      when $f_j$ is as above. (The last choice is possible by Lemma~\ref{lem:approx-iterates}.)

    This completes the recursive construction. As before, the functions $f_j$ converge
     to a meromorphic function $f$ with the desired properties.        
\end{proof}

\begin{rmk}\label{rmk:oscillating-shrinking}
    In Theorem~\ref{thm:Jordan-continua merged-oscillating}, 
    we may ensure additionally that 
    \begin{equation}
      \label{eqn:oscillating-shrinking}
      \lim_{n\to\infty} \max_{0<k<k_n} \diam f^k(\theta(X_n)) = 0. 
    \end{equation}
    (This  will be important in the proof of Corollary~\ref{cor:Jordan}.)

   Indeed, in the proof of the theorem, we may 
     choose the discs $D_n$ so that their Euclidean diameter
    tends to zero as $n\to\infty$. We may also assume that
    $\lvert \psi_n'\rvert\geq 2$ for all $n\geq 0$, and that
    the $\alpha_n$ are chosen so small that the approximating function $f$
    satisfies $\lvert f'\rvert \geq 1$ on
    $\overline{\Sigma_n}$.

    Suppose that we do so. Recall that, for $0<k<k_n-1$, the set 
    $f^k(\theta(X_n))$ is contained in a connected component
    of $f^{-1}(D_k)$ that is mapped univalently onto $D_k$ and is contained in
    $\Sigma_{k-1}$. Our assumption means that $\lvert f'\rvert \geq 1$ on this component,
    and we conclude inductively that
      \[ \diam f^k(\theta(X_n)) \leq \diam D_{k_n-2} \]
    for every $n\geq 0$ and $0<k<k_n$. By assumption $\diam D_{k_n-2}=\diam D_{n}\to 0$ as $n\to\infty$,
    establishing~\eqref{eqn:oscillating-shrinking}.
\end{rmk}

To prove Theorem~\ref{thm:entire}, we use the following fact, which follows
 from standard separation theorems.
\begin{prop}\label{prop:topology}
  Suppose that $(K_n)_{n=0}^{\infty}$ is a sequence of pairwise disjoint
   non-empty and non-separating compact sets 
   $K_n\subset \C$ such that $\min_{z\in K_n}\lvert z\rvert\to\infty$
   as $n\to\infty$. 

  Then there exists an increasing sequence $(\Lambda_n)_{n=0}^{\infty}$ of closed bounded
  Jordan domains such that $K_m\subset \Lambda_n$ for $m<n$, $K_m\cap \Lambda_n=\emptyset$ for
  $m\geq n$, and such that $\C = \bigcup_{n=0}^{\infty} \Lambda_n$.   
\end{prop}

\begin{proof}[Proof of Theorem~\ref{thm:entire}]
 Recall that the theorem claims that, in the case where $\C\setminus X$ is connected, the 
  function obtained in either of the two preceding proofs can be chosen
  to be entire. In other words, we must show that the functions 
  $f_{j}$ can be chosen to be polynomials, which by Runge's theorem is the
  case if the sets $A_j$, defined in~\eqref{eqn:Ajescaping} and~\eqref{eqn:Ajoscillating},
   respectively, also do not separate the plane. 

  This is easy to achieve. Indeed, since $X_j$ does not separate the plane, we may choose
   the open neighbourhoods $W_j$ so that $\overline{W_j}$ does not separate
   the plane; see~\cite[Lemma~2.9]{mrw1}. The same is true for the open neighbourhood $U_{j-1}$ chosen as
   part of the recursive construction. The sets $Q_{j-1}$ may be chosen to be finite;
   $R_{j-1}$ can then be chosen to be a finite disjoint union of round closed discs. 
   The sets $\Omega_{j-1}$ and $\Sigma_{j-1}$ are chosen to be discs
       Theorem~\ref{thm:Jordan-continua merged-oscillating}.
   Finally, it follows from Proposition~\ref{prop:topology} that the sets
   $\Lambda_n$ can also be chosen so that they  do not separate the plane; indeed,
   $\Lambda_n$ may be chosen to be a closed Jordan domain.
\end{proof}

\begin{rmk}\label{rmk:Cm}
    In all of the above proofs, we have chosen $M_n=0$ for all $n\geq 0$, but we can
    also choose larger values of $M_n$. We obtain a map $\theta$ that is
    $\calC^{M_n}$-close to the identity on a neighbourhood of $X_n$, and equal to
    the identity elsewhere. Moreover, recall that we chose the approximating function
    $f_j$ to agree with $g_j$ on $\bigcup_{k=0}^{j-1} C_k$; we may furthermore choose it to
    agree with $g_n$ up to order $M_n$. Then the same is true for the approximating
    function $f$, and it follows from the proof of Theorem~\ref{thm:triangle1} that
    $\theta$ agrees with $\id$ to order $M_n$ on the points of $C_n$. This justifies
    the claims made in Remark~\ref{rmk:Cm-intro}.
\end{rmk}

\section{The univalent case}\label{sec:univalent}

In this section, we briefly discuss the proof of Theorem \ref{thm:injective}. We adapt Theorems~\ref{thm:triangle1} and~\ref{thm:triangleaddendum} to the univalent case as follows. 

\begin{thm}\label{thm:injectivetriangle}
 Let $\Phi = \bigl((X_n)_{n=0}^{\infty},(\phi_n)_{n=0}^{\infty},(C_n)_{n=0}^{\infty}\bigr)$ be a non-autonomous system, 
  satisfying the condition that $\phi_n$ is univalent with $\phi_n(X_n)=X_{n+1}$ for all $n\geq 0$.
      Let $(\varepsilon_n)_{n=0}^\infty$ be a sequence of positive real numbers and let
   $(M_n)_{n=0}^{\infty}$ be a sequence of non-negative integers. 

     Then there are sequences $(\delta_j)_{j=1}^{\infty}$ and $(\hat{\delta}_j)_{j=1}^{\infty}$
       of functions 
           \[ \delta_j  \colon \Delta_{j-1}(\Phi) \times \U_{j-1}(\Phi)\to (0,\infty)\quad\text{and}\quad
            \hat{\delta}_j \colon \Delta_{j}(\Phi) \to (0,\infty)
           \]
       with the following property. 

   Suppose that $\T=(f_{n,j})_{0\leq n<j<\infty}$ is an infinite triangle satisfying~\eqref{eqn:fclose} 
    for all $j\geq 1$, and that the following hold for every $j\geq 1$.
    \begin{enumerate}[(a)]
      \item The function 
         \[ \theta_{j-1} \defeq f_{j-2,j-1}\circ\dots \circ f_{0,j-1}\circ 
                            \phi_0^{-1}\circ \dots \circ \phi_{j-1}^{-1}\]
             is defined and univalent in a neighbourhood of $X_{j-1}$.
             (By convention, $\theta_{0}=\id$.)
       \item $\theta_{j-1}^{-1}$ is defined on $U_{j-1}(\T)$.
       \item $\lvert f_{j-1,j} - \phi_{j-1} \circ \theta_{j-1}^{-1}\rvert < 
                  \delta_{j}(\T_{j-1},U_{j-1}(\T))$.
    \end{enumerate}

  Then the conclusions of Theorems~\ref{thm:triangle1} and~\ref{thm:triangleaddendum} hold, where
  additionally we can take $\Theta_0=\id$.
\end{thm}
\begin{proof}The proof proceeds as 
 for Theorems~\ref{thm:triangle1} and~\ref{thm:triangleaddendum}, except that 
 the maps $\Theta_{n,j}$ and $\theta_{n,j}$ that appear in the proofs are
 constructed slightly differently, as follows. (See Figure~\ref{fig:univalent}.) 

 We set $\Theta_{0,j}\defeq \theta_{0,j}\defeq \id$
  for all $j\geq 0$. 
  %On the other hand, we no longer
 % have $\Theta_{j,j}=\id$ for $j\geq 1$;
 % instead, $\Theta_{j,j}$ agrees with the map
 % $\theta_j$ from the statement of the theorem
 % in a neighbourhood of $X_j$. 
 By Lemma~\ref{lem:univ-approx}, if $\delta_j$ is chosen sufficiently small at each step, the 
 assumption ensures that the maps
 $f_{n,j}$ are all  univalent on 
 an open set slightly smaller than $U_n(\T)$. Likewise, each map
 $\theta_n$ from the statement of the theorem is defined and univalent on 
 a neighbourhood of $X_n$. (Here we also use Lemma~\ref{lem:approx-iterates}.)
 For $1\leq n \leq j$, we inductively define $\theta_{n,j}$ (on a suitable neighbourhood 
 $V_{n,j}$ of $X_{n,j}$) by 
 $\theta_{n,j}\defeq f_{n-1,j+1}\circ \theta_{n-1,j}
   \circ f_{n-1,j}^{-1}$.  Clearly if $\delta_{j+1}$ is
   chosen sufficiently small, these maps
   are close to the identity, as is $\theta_{j+1}$. As in 
   the proof of Theorem~\ref{thm:triangle1}, we define 
   $\Theta_{n,j+1}\defeq \tilde{\theta}_{n,j}\circ \Theta_{n,j}$, where
   $\tilde{\theta_{n,j}}$ is obtained from $\theta_{n,j}$ via Corollary~\ref{cor:Cmextension}.
   Similarly, we obtain $\Theta_{j+1,j+1}$ by applying Corollary~\ref{cor:Cmextension} to
   $\Theta=\id$ and $\theta = \theta_{j+1}$.
\end{proof}

%Theorem \ref{thm:injectivetriangle} is the main ingredient in the proof of Theorem \ref{thm:injective}, which we sketch below.
\begin{proof}[Proof of Theorem \ref{thm:injective}]
The proof proceeds analogously to the proof of Theorems~\ref{thm:Jordan-continua merged} and~\ref{thm:Jordan-continua merged-oscillating}, with the exception of the choice of $U_{j-1}$ and the map $g_j$ on $\overline{U_{j-1}}$, which are chosen precisely
so that we can apply Theorem~\ref{thm:injectivetriangle}.
\end{proof}

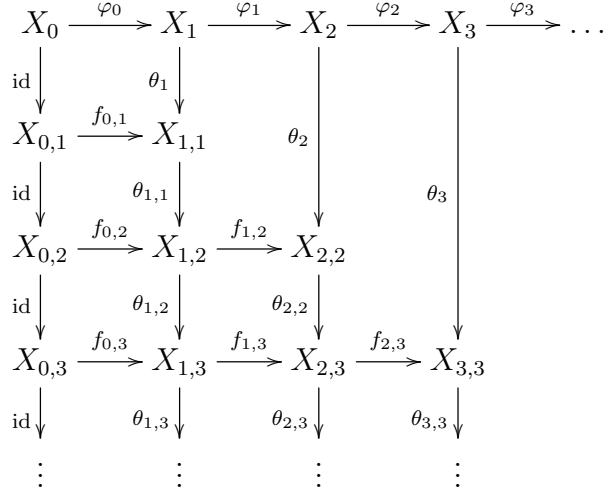
\begin{figure} 
\[
\xymatrix{
X_0 \ar[d]_{\textup{id}} \ar[r]^{\varphi_0} & X_1 \ar[d]_{\theta_1} \ar[r]^{\varphi_1} &  X_2 \ar[dd]_{\theta_2} \ar[r]^{\varphi_2} & X_3 \ar[ddd]_{\theta_3} \ar[r]^{\varphi_3} & \dots\\
X_{0,1} \ar[d]_{\textup{id}} \ar[r]^{f_{0,1}} & X_{1,1} \ar[d]_{\theta_{1,1}} &  &   \\
X_{0,2} \ar[d]_{\textup{id}} \ar[r]^{f_{0,2}} & X_{1,2} \ar[d]_{\theta_{1,2}} \ar[r]^{f_{1,2}} & X_{2,2} \ar[d]_{\theta_{2,2}} & \\
X_{0,3}  \ar[d]_{\textup{id}} \ar[r]^{f_{0,3}}  & X_{1,3}\ar[r]^{f_{1,3}}  \ar[d]_{\theta_{1,3}} & X_{2,3}\ar[r]^{f_{2,3}}  \ar[d]_{\theta_{2,3}} & X_{3,3}  \ar[d]_{\theta_{3,3}}\\
\vdots & \vdots & \vdots & \vdots
}
\]
\caption{Diagram for the proof of Theorem~\ref{thm:injective}. On $X_n$, the map
 $\Theta_{n,j}$ is the composition $\theta_{n,j-1}\circ \dots \circ \theta_{n,n}\circ \theta_n$. Note that $\Theta_{0,j}=\id$ for all $j\geq 0$.}\label{fig:univalent}
\end{figure}

\section{Realising internal dynamics by Jordan wandering domains}
\label{sec:topology-wandering-domains}
In this section, we prove Corollary~\ref{cor:Jordan}. We begin by stating 
 the two classification theorems, which describe the different types of simply connected wandering domains. The first classification (see \cite[Theorem A]{benini-fagella-evdoridou-rippon-stallard}) is in terms of hyperbolic distances between orbits of points.

\begin{thm}[First classification theorem] \label{thm:Theorem A introduction}
Let $U$ be a simply connected wandering domain   of a transcendental entire function $f$ and let $U_n$ be the Fatou component containing $f^n(U)$, for $n \ge 0$. Define the countable set of pairs
\[
E=\bigl\{(z,z')\in U\times U : f^k(z)=f^k(z') \text{\ for some $k\in\N$}\bigr\}.
\]
Then, exactly one of the following holds.
\begin{enumerate}[(1)]
\item \label{item:first-type1}$\dist_{U_n}(f^n(z), f^n(z'))\to c(z,z')= 0 $ for all $z,z'\in U$, and we say that $U$ is {\em (hyperbolically) contracting};
\item \label{item:first-type2}$\dist_{U_n}(f^n(z), f^n(z'))\to c(z,z') >0$ and $\dist_{U_n}(f^n(z), f^n(z')) \neq c(z,z')$
for all $(z,z')\in (U \times U) \setminus E$, $n \in \N$, and we say that $U$ is {\em  (hyperbolically) semi-contracting}; or
\item \label{item:first-type3}there exists $N>0$ such that for all $n\geq N$, $\dist_{U_n}(f^n(z), f^n(z')) = c(z,z') >0$ for all $(z,z') \in (U \times U) \setminus E$, and we say that $U$ is {\em (hyperbolically) eventually isometric}.
\end{enumerate}
\end{thm}
The second classification (see \cite[Theorem C]{benini-fagella-evdoridou-rippon-stallard}) is in terms of convergence to the boundary (in Euclidean distance).

\begin{thm}[Second classification theorem]\label{thm:classification2}
Let $U$ be a simply connected wandering domain of a transcendental entire function~$f$ and let $U_n$ be the Fatou component containing $f^n(U)$, for $n \ge 0$. Then exactly one of the following holds:
\begin{enumerate}[(a)]
\item \label{item:second-typeA}
  $\liminf_{n\to\infty} \operatorname{dist}(f^{n}(z),\partial U_{n})>0$  for all $z\in U$,
    that is, all orbits stay away from the boundary;
\item \label{item:second-typeB}there exists a subsequence $n_k\to \infty$ for which $\operatorname{dist}(f^{n_k}(z),\partial U_{n_k})\to 0$ for all $z\in U$, while for a different subsequence $m_k\to\infty$ we have that
    \[\liminf_{k \to \infty} \operatorname{dist}(f^{m_k}(z),\partial U_{m_k})>0, \quad\text{for }z\in U;\]
    %that is, the orbits have a bungee behaviour towards the boundary.
\item \label{item:second-typeC}$\operatorname{dist}(f^{n}(z),\partial U_{n})\to 0$ for all $z\in U$, that is, all orbits converge to the boundary.
\end{enumerate}
\end{thm}
These two classifications give rise to nine types of simply connected wandering domains, only six of which can be realised as oscillating wandering domains, since we cannot have the case where orbits stay away from the boundary.

We will use the following corollary, which follows easily from 
Theorems \ref{thm:entire} and \ref{thm:injective} and Remark~\ref{rmk:Cm}.

\begin{cor}\label{cor:examples}
    For $n\geq 0$, let $D_n= \{z\colon|z-4n|<1\}$ and let $\varphi_n\colon\overline{D_n} \to \overline{D_{n+1}}$ be holomorphic in a neighbourhood of $\overline{D_n}$ with $\phi_n(\partial D_n)= \partial D_{n+1}$. Let also $(\varepsilon_n)_{n=0}^{\infty}$ be a sequence of
positive numbers such that $\varepsilon_n \to 0$ as $n \to \infty$. Then there exists a transcendental entire function $f$ with a sequence of escaping wandering domains $(U_n)_{n=0}^\infty$ and a diffeomorphism $\theta\colon \mathbb{C} \to \mathbb{C}$ such that, for all $n\geq 0$,
\begin{enumerate}[(a)]
\item $\theta(\overline{D_n}) = \overline{U_n}$;% for all $n\geq 0$;
\item $f \circ \theta = \theta \circ \varphi_n$ on $\overline{D_n}$;% for all $n\geq 0$;
\item $\|\theta-\id\|_{\calC^2} \leq \varepsilon_n$ on $\overline{D_n}$;% for all $n\geq 0$;
\item $\theta$ is conformal on $D_n$;% \dmp{for all $n\geq 0$}; 
\item $f^n(0)=\Phi_n(0)$ and $f^n(1) = \Phi_n(1)$, where $\Phi_n= \phi_{n-1} \circ \cdots \circ \phi_0$;%, \dmp{for all $n\geq 0$};
\item $\Deriv \theta(\Phi_n(1)) = \Deriv \id(\Phi_n(1))$.% \dmp{for all $n\geq 0$}.
\end{enumerate}
If, furthermore, the maps $(\varphi_n)_{n=0}^\infty$ are univalent, then $U_0=\mathbb{D}$.
\end{cor}

\begin{proof}[Proof of Corollary~\ref{cor:Jordan}]

Let $(D_n)_{n=0}^\infty$ be as in Corollary \ref{cor:examples} and for $n\geq 0$, take $\phi_{n} = b_{n+1}(z-4n)+4(n+1)$, where is a finite Blaschke product.

We first treat the case of escaping wandering domains. Take $(b_n)_{n=0}^\infty$ to be as in 
one of the nine examples provided in \cite[Section 6]{benini-fagella-evdoridou-rippon-stallard}. Then we apply Corollary~\ref{cor:examples} %\dmp{Theorem \ref{thm:entire}} 
to obtain a transcendental entire function $f$ and a diffeomorphism $\theta$. Now $f$ has an orbit of escaping wandering domains $(U_n)_{n=0}^\infty$ such that $U_n=\theta(D_n)$ for $n\geq 0$. 
For all $z\in D_0=\DD$ and $n\geq 0$,
\begin{align*} &\lvert \dist(f^n(\theta(z)),\partial U_n) - \dist(\Phi_n(z),\partial D_n)
 \rvert   \\ & 
   \hspace{80pt}=\lvert \dist(\theta(\phi^n(z)),\partial \theta(D_n)) - \dist(\Phi_n(z), \partial D_n)
      \rvert \leq 2\eps_n,\end{align*}
which tends to zero as $n\to\infty$.
Also, $\theta$ is conformal on each $D_n$, and hence preserves hyperbolic distances. We conclude that $(f^n)_{n=1}^{\infty}$ and $(\Phi_n)_{n=1}^{\infty}$ have the same internal dynamics on $U_0$ and $\DD$ respectively. Now each of the nine choices of the
maps $(b_n)_{n=0}^\infty$ in \cite[Section 6]{benini-fagella-evdoridou-rippon-stallard} realises
one of the nine types of internal dynamics, and hence the conclusion follows.

For the case of oscillating wandering domains, we use Theorem~\ref{thm:entire} and Remark \ref{rmk:oscillating-shrinking}. Let $(b_n)_{n=0}^\infty$ be defined as in \cite[Section~5]{evdoridou-rippon-stallard}, $\phi|_{\overline{D_n}}=\phi_n$ and $\varepsilon_n \to 0$ as $n \to \infty$. We apply Theorem~\ref{thm:entire} to obtain a transcendental entire function $f$ with an orbit of oscillating wandering domains $(U_n)_{n=0}^\infty$ such that the conclusions of Theorem \ref{thm:Jordan-continua merged-oscillating} hold. 
As above, for all $z\in D_0=\DD$ and $n\geq 0$ we have
\[\lvert \dist(f^{k_n+\cdots+k_1}(\theta(z)),\partial U_{k_n}) - \dist(\Phi_n(z),\partial D_n)
 \rvert \leq 2\eps_n,\]
which tends to zero as $n\to\infty$.  It follows by Remark~\ref{rmk:oscillating-shrinking} that when $\Phi^n(0)=4n$ then $U_n$ will be of type \ref{item:second-typeB} from Theorem~\ref{thm:classification2} and when $\Phi_n(0) \to 4n+1$ then $(U_n)_{n=0}^\infty$ will be of type \ref{item:second-typeC} from Theorem~\ref{thm:classification2}.

Also, for all $z, w \in D_0=\DD$ and $n\geq 0$, \[\dist_{U_{k_n}}(f^{k_n+ \cdots +k_1}(\theta(z)), f^{k_n+\cdots+k_1}(\theta(w)))=\dist_{D_n}(\Phi_n(z),\Phi_n(w)).\] Hence the internal dynamics with respect to hyperbolic distance is preserved. This concludes the proof.
\end{proof}

\section{Wandering domains with interesting boundary dynamics} \label{sec:examples}

As mentioned in the introduction, in \cite{benini-fagella-evdoridou-rippon-stallard2} the authors
 construct a number of interesting examples of boundary behaviour for non-autonomous
 sequences of maps on simply connected domains, and ask whether these can also be
 realised by wandering domains of entire functions. We show that this is indeed the case
 for all of the examples considered. In most cases (namely in 
 \cite[Theorem~D, Example 7.6 and Examples 8.1--8.3]{benini-fagella-evdoridou-rippon-stallard2}),
 the non-autonomous model consists of Blaschke products of the unit disc, and hence
 the realisation is a simple application of Corollary~\ref{cor:examples}.  On the other
 hand, \cite[Example~4.6]{benini-fagella-evdoridou-rippon-stallard2} involves model
 maps that do not extend holomorphically beyond the boundary of their domains, and
 hence we must adjust the construction somewhat (see Example~\ref{ex:measurezero} below).

We will use the following definition (see \cite[Definition~1.3]{benini-fagella-evdoridou-rippon-stallard2}).
\begin{dfn}\label{DWset}
Let $F_n\colon U \to U_n, n\geq 1$ be a sequence of holomorphic maps between simply connected domains. Assume that each $F_n$ extends continuously to $\partial U$, and
such that interior orbits converge to the boundary (i.e., $\dist(F_n(z),\partial U)\to 0$ as $n\to\infty$ for all $z\in U$).

The {\it Denjoy--Wolff set} of $(F_n)_{n=1}^{\infty}$ consists of those points $\zeta \in \partial U$ such that, for all $z \in U$,
\[
|F_n(\zeta) - F_n(z)| \to 0\quad \textup{as } n \to \infty.
\]
\end{dfn}
\begin{rmk}
    In~\cite{benini-fagella-evdoridou-rippon-stallard2}, it is not assumed that
    the maps $F_n$ are defined on the boundary of $U$, in which case the definition
    is made using a radial extension (see \cite[Section 2.1]{benini-fagella-evdoridou-rippon-stallard2}). However, in this section we only consider maps that extend continuously
    to the boundary.
\end{rmk}

\begin{ex}\label{ex:emptyDW}
There exists a transcendental entire function $f$ with an orbit of bounded escaping wandering domains $(U_n)_{n=0}^{\infty}$ such that $\operatorname{dist}(f^n(z), \partial U_n) \to 0$ for all $z \in U_0$ but the Denjoy--Wolff set of $(f^n|_{U_0})_{n=1}^{\infty}$ is empty.
\end{ex}
\begin{proof}
In the proof of \cite[Theorem~D]{benini-fagella-evdoridou-rippon-stallard2},
 the authors construct 
 a sequence $(\mu_n)_{n=1}^{\infty}$ of M\"obius automorphisms of the unit disc such that the
 compositions $M_n= \mu_n \circ \cdots \circ \mu_1$ have the following properties. \begin{itemize}
\item[(a)] $M_n(z)\to 1$ for all $z\in \mathbb{D}$.
\item[(b)] There is an arc $S=(e^{-i\theta}, e^{i\theta})$, $\theta \in (0, \pi/4)$, around $1$ such that for all $\zeta \in \partial \mathbb{D}$, there are infinitely many $n$
  with $M_n(\zeta)\notin S$.
\end{itemize}

We now apply Corollary \ref{cor:examples} with $\phi_n(z)\defeq \mu_{n+1}(z-4n)+4(n+1)$
and any sequence $(\varepsilon_n)_{n=0}^\infty$ such that
$\eps_n\to 0$ as $n\to \infty$. Let $f$ and $\theta$ be the resulting  
 transcendental entire function and $\calC^{\infty}$ diffeomorphism. Set
 $U_n\defeq \theta(\DD+4n)$ for $n\geq 0$. Then $(U_n)_{n=0}^{\infty}$ is
 an orbit of simply connected wandering domains, with $U_0=\DD$, and, for all $z \in \mathbb{D}$, 
 \[ \lvert f^n(z) - 4n - M_n(z)\rvert = \lvert \theta(M_n(z)+4n)- (M_n(z)+4n)\rvert
        < \eps_n \to 0\quad \textup{as } n\to\infty. \]
 For $z\in U_0$, we have 
   \[ \dist(f^n(z),\partial U_n) \leq \lvert f^n(z)-f^n(1)\rvert 
       \leq \lvert M_n(z) - 1\rvert + 2\eps_n \to 0 \quad \textup{as } n\to\infty. \]
 On the other hand, suppose that $z\in U_0$,
 $\zeta\in \partial\DD$ and $n$ is such that
 $M_n(\zeta)\notin S$. Then
   \begin{align*} \lvert f^n(\zeta) - f^n(z)\rvert &\geq 
       \lvert M_n(\zeta) - M_n(z)\rvert - 2\eps_n \geq
       \lvert M_n(\zeta) - 1\rvert - \lvert M_n(z) - 1\rvert - 2\eps_n \\ &\geq
       \diam(S) - \lvert M_n(z) -1 \rvert - 2\eps_n \to \diam(S) > 0 \end{align*}
  as $n\to\infty$. By construction, there are infinitely many $n$ with this property,
   and hence $\limsup_{n\to\infty} \lvert f^n(\zeta) - f^n(z)\rvert > 0$ as desired.\end{proof}

The next example shows that the harmonic measure
of the Denjoy--Wolff set may take a value strictly between $0$ and $1$. (This is in contrast
to the autonomous case, i.e.\ the iteration of a single self-map of a simply connected domain \cite[Theorems 4.1 and 4.2]{doering-mane}; see also \cite[Theorem~1.2]{benini-fagella-evdoridou-rippon-stallard2}.)

\begin{ex}
There exists a transcendental entire function $f$ with an orbit of bounded escaping eventually isometric wandering domains $(U_n)_{n=0}^\infty$ such that $\operatorname{dist}(f^n(z), \partial U_n) \to 0$ as $n\to\infty$ for all $z \in U_0$ and the Denjoy--Wolff set of $(f^n|_{U_0})_{n=0}^\infty$
has positive but not full harmonic measure with respect to $U_0$.
\end{ex}
\begin{proof}
In \cite[Example~7.6]{benini-fagella-evdoridou-rippon-stallard2} the authors show that 
there exists a sequence $(\mu_n)_{n=1}^{\infty}$ of M\"obius automorphisms of the unit disc such that $\mu_n(1)=1$ for all $n\geq 0$ and the
 compositions $M_n= \mu_n \circ \cdots \circ \mu_1$ have the following properties:
\begin{enumerate}[\rm (a)]
\item $M_n(z)\to 1$ as $n\to\infty$, for all $z\in\DD$;
\item $M_n(\zeta)\to 1$ as $n\to\infty$, for all $\zeta=e^{i\theta}$, $|\theta|<\pi/2$; but
\item $M_n(\zeta) \not\to 1$ as $n\to\infty$, for all $\zeta=e^{i\theta}$, $\pi/2 \leq |\theta| \leq \pi$.
\end{enumerate}

As in the previous example, we apply Corollary~\ref{cor:examples} to obtain %take $X_n= \overline{D_n}= \{z: |z-4n|\leq 1\}$ and $\phi= T_{n+1} \circ \mu_n \circ T_n ^{-1}$, where $T_n(z)= z+4n$, $n \geq 0$. We apply Theorem \ref{thm:injective} to obtain 
a transcendental entire function $f$ with an orbit of simply connected wandering domains $(U_n)_{n=0}^{\infty}$ such that $U_0= \mathbb{D}$ and 
 $\lvert f^n(z)-4n-M_n(z)\rvert \leq \eps_n\to 0$ as $n\to \infty$ for all $z \in \overline{\mathbb{D}}$. It follows that
 $\lvert f^n(\zeta) - f^n(z)\rvert \to 0$ when 
 $z\in \DD$ and $\zeta=e^{i\theta}$, $|\theta|<\pi/2$, while
 $\lvert f^n(\zeta)-f^n(z) \rvert \not\to 0$ when $\zeta=e^{i\theta}$, $\pi/2\leq |\theta|\leq \pi$.\end{proof}

The final two examples concern a relation between the rate of convergence to the
boundary for interior points and the size of the Denjoy--Wolff set. By~\cite[Theorem~C]{benini-fagella-evdoridou-rippon-stallard2}, if $(U_n)_{n=0}^{\infty}$ is an orbit of
bounded simply connected wandering domains of a transcendental entire function $f$,
and if 
 \begin{equation}\label{eqn:convergentsum} \sum_{n=0}^{\infty} \dist(f^n(z_0),\partial U)^{1/2} < \infty,\end{equation}
 then almost all points in $\partial U_0$ (with respect to harmonic measure)
 belong to the Denjoy--Wolff set. The next example shows that there is no rate of
 convergence to the boundary that ensures that the Denjoy--Wolff set has zero measure.
(Again, this is in contrast to the autonomous case \cite{doering-mane,benini-fagella-evdoridou-rippon-stallard2}.)
 
\begin{ex}
Let $(\delta_n)_{n=1}^{\infty}$ be a sequence of positive numbers with $\delta_n\leq 1$ for $n\geq 1$ and
$\delta_n\to 0$ as $n\to\infty$. Then there
exists a transcendental entire function $f$ such that $U_0=\DD$ is 
a bounded escaping and eventually isometric wandering domain of $f$ whose
orbit $(U_n)_{n=0}^{\infty}$ has the following properties.
\begin{itemize}
\item $\dist(f^n(0),\partial U_n) = \delta_n$ for all $n\geq 1$.
\item Every point $\zeta\in\partial \DD\setminus \{-1\}$ is in the Denjoy--Wolff set of
     $(f^n|_{\DD})_{n=1}^{\infty}$. 
\end{itemize}
\end{ex}
\begin{proof}
Set
$a_0\defeq 0$ and $a_n\defeq 1-\delta_n$; then $a_n\to 1$ as $n\to\infty$. In \cite[Example 8.1]{benini-fagella-evdoridou-rippon-stallard2} the authors construct the following example.

For $n\geq 1$, let $\mu_n$ be the M\"obius map such that $\mu_n(\pm 1)=\pm 1$ and $\mu_n(a_{n-1})=a_n$, and define $M_n=\mu_n\circ\cdots\circ \mu_1$. Then
\[M_n(z)=\frac{z+a_n}{1+a_nz}, \quad \text{for } n \geq 1.\]
Then every  point $\zeta \in \partial \mathbb{D}\setminus\{-1\}$ satisfies that $\lim_{n\to\infty} M_n(\zeta)=1$.
%%, regardless of whether $\sum_{n\geq 0} (1-a_n)$ converges or not.

Apply Corollary~\ref{cor:examples}, as in the previous two examples.
We obtain a transcendental entire function $f$ and an orbit $(U_n)_{n=0}^{\infty}$ of simply connected wandering domains 
with $U_0=\DD$. We have $\lvert f^n(z)-(1+4n)\rvert \to 0$ as $n\to\infty$ for all
$z\in \overline{\DD}\setminus\{-1\}$, and hence the Denjoy--Wolff set is
$\partial\DD\setminus\{-1\}$. Recall that the map $\theta$ 
from Corollary~\ref{cor:examples} agrees with the identity to first order at 
$\Phi_n(1)=1 + 4n$,
and is arbitrarily close to the identity in the $\calC^2$-norm. It follows easily that,
 if $\eps_n$ is chosen sufficiently small, $f^n(1)=1+4n$ is the closest point of
 $\partial D_n$ to $f^n(0)=\Phi_n(0)=1-\delta_n+4n$. 
% We now take $\sum_{n\geq 0} (1-a_n)= \infty$ and  apply Corollary~\ref{cor:examples} to obtain %consider $X_n= \overline{D_n}= \{z: |z-4n|\leq 1\}$ and $\phi= T_{n+1} \circ \mu_n \circ T_n ^{-1}$, where $T_n(z)= z+4n$, $n \geq 0$. We apply Theorem \ref{thm:injective} to obtain 
% a transcendental entire function $f$ with a sequence of simply connected wandering domains $U_n$ such that $U_0= \mathbb{D}$ and $U_n$ is asymptotically equal to $D_n$, such that $|f-\phi|<\delta_n$ on $\overline{D_n}$. It follows that all points in $U_0$ tend to a point $\zeta_n=f^n(1)$ of $\partial U_n$. Moreover, $\delta_n$ can be chosen so that $f^n(\zeta) \to \zeta_n$, for a long arc on $\partial U_0$, deducing that $\lim_{n \to \infty} f^n(z)=\zeta_0$ for $z \in U_n \cap  D(\zeta_0, \rho)$, $<0<\rho<1$.
\end{proof}

\begin{rmk}
    As remarked after~\cite[Example~8.1]{benini-fagella-evdoridou-rippon-stallard2}, the
     same effect can be achieved by using a sequence of Blaschke products of higher degree,
     instead of M\"obius transformations. (Simply choose a Blaschke product $b_n$ that
     is very close to $\mu_n$ outside a small disc centred at $-1$.) Applying the
     same proof as above, we also obtain
     examples of semi-contracting wandering domains with arbitrarily 
     slow convergence to the boundary and such that the Denjoy--Wolff set has 
     full harmonic measure. (Note that we will not have $U_0=\DD$ 
     in this case, since the model map is no longer univalent.) We omit the details.
\end{rmk}

Our final example shows that the condition~\eqref{eqn:convergentsum} in
\cite[Theorem~C]{benini-fagella-evdoridou-rippon-stallard2} is
optimal, even for wandering domains of entire functions.

\begin{ex}\label{ex:measurezero}
There exists a transcendental entire function $f$ with an orbit of bounded, escaping 
and contracting wandering domains $(U_n)_{n=0}^{\infty}$ such that
\begin{enumerate}[(a)]
\item each $U_n$ is a Jordan domain with smooth boundary;
\item \label{item:convergestoboundary} for all $z\in U_0$, 
  $\dist(f^n(z),\partial U_n)\to 0$ as $n\to\infty$; %f^n(z)\to \zeta_0 \in \partial U$ as $n \to \infty$;\label{item:convergestoboundary}
\item \label{item:seriesconvergence}
  for all $z\in  U_0$, $\sum \operatorname{dist}(f^n(z), \partial U_n)^{\alpha}
\begin{cases}
< \infty,\quad \text{ for}\quad \alpha > 1/2,\\
= \infty,\quad \text{ for}\quad \alpha = 1/2;
\end{cases}$
\item the Denjoy--Wolff set of $(f^n|_{U_0})_{n=1}^{\infty}$ has 
harmonic measure zero seen from $U_0$.
\end{enumerate}
\end{ex}
\begin{proof}
We follow the idea of \cite[Example~4.6]{benini-fagella-evdoridou-rippon-stallard2}, 
with some adjustments. For $n\geq 1$, let
\[
M_{n}(z)\defeq\frac{z+1- \frac1n}{1+\left(1-\frac1n\right)z}.
\]
Define $b_1(z)\defeq M_1(z)=z$, while for $n\geq 2$, we define the Blaschke product
\[
b_n(z)\defeq M_n\left((M_{n-1}^{-1}(z))^2\right),
\]
and
\begin{equation}\label{Gn-Mn}
B_n(z)\defeq b_n\circ \cdots \circ b_1 (z)
= M_n(z^{2^n})=\frac{z^{2^n}+1- \frac1n}{1+(1-\frac1n)z^{2^n}}.
\end{equation}
As shown in \cite[Example 8.3]{benini-fagella-evdoridou-rippon-stallard2}, for all $z\in \mathbb{D}$, $B_n(z)\to 1$ as $n\to \infty$ and 
$\{B_n(\zeta)\}_{n=1}^{\infty}$ is dense in $\partial\DD$ for almost every point of $\zeta\in\partial\mathbb{D}$.

Define
 \[ \psi\colon \DD\to \C; \quad z\mapsto (z-1)^2 \]
 and take $U\defeq  \psi (\mathbb{D})$. Then
$\psi\colon \DD\to U$ is a conformal isomorphism.  In~\cite[Example~4.6]{benini-fagella-evdoridou-rippon-stallard2}, the authors consider the self-maps of $U$ given by
$\psi\circ B_n\circ\psi^{-1}$. Since $\psi$ does not extend holomorphically to a neighbourhood of $0$, we instead use a sequence of smaller domains for our
model maps. For $n\geq 1$, define 
\[ \psi_n\colon \DD\to U; \quad z\mapsto \psi((1-\tfrac{1}{n+1})z)+ 8n\]
 and $X_n\defeq \overline{\psi_n(\DD)}$. Also set $X_0\defeq \overline{\DD}$ and 
  $\psi_0=\id$. Define
  \[ \phi(z) \defeq \psi_{n+1}\circ b_{n+1}\circ \psi_n^{-1}\]
  on $X_n$. Note that this means that
    \[ \phi^n(z) = \psi_{n}(B_n(z)) \]
    for all $z\in\DD$, and that $\phi$ extends analytically to a neighbourhood of 
   $X \defeq \bigcup_{j=0}^{\infty}X_j$ by definition.

Set 
$V_n\defeq \interior(X_n)$ for $n\geq 1$. We have 
    \begin{align*}
    \dist(\phi^n(0),\partial V_n)
       \leq \phi^n(0) - 8n &= \psi_n(B_n(0)) - 8n \\ &= \psi((1-\tfrac1n)(1-\tfrac{1}{n+1})) 
      = \bigl(-\tfrac1n - \tfrac{1}{n+1} + \tfrac1{n(n+1)}\bigr)^2
      \asymp \tfrac{1}{n^2}.\end{align*}
   From the same calculation, we see that $\psi'(B_n(0)) \asymp \tfrac{1}{n}$; from Koebe's
   theorem we conclude that
   $\dist(\phi^n(0),\partial V_n) \asymp 1/n^2$, and hence $\phi$ has the property stated 
   in~\ref{item:seriesconvergence}. Also recall that, for almost every point $\zeta\in\partial \DD$,
   there is a sequence $(n_k)_{k=0}^{\infty}$ with $n_k\to\infty$ as $k\to\infty$ such that
   $B_{n_k}(\zeta)\to -1$, and hence 
      \begin{equation}\label{eqn:notDW} \phi^{n_k}(\zeta) - 8n \to 4\end{equation} as $k\to\infty$.

Now apply Theorem \ref{thm:entire} with $C_n = \{\phi^n(0)\}$ and
  $\varepsilon_n = o(1/n^2)$ as $n \to \infty$.
We obtain a transcendental entire function $f$ with an orbit $(U_n)_{n=0}^{\infty}$ of simply connected escaping wandering domains  and a diffeomorphism $\theta$ such that 
  \[ \lvert f^n(\theta(z)) - \phi^n(z)\rvert =
   \lvert \theta(\phi^n(z)) - \phi^n(z)\rvert \leq \eps_n \]
   for all $z\in X_0$. We have 
   \[ \lvert \dist(f^n(\theta(0)),\partial U_n) - \dist(\phi^n(0),\partial V_n)\rvert \leq 
       2\eps_n = o\bigl(\tfrac{1}{n^2}\bigr). \]
    So $\dist(f^n(\theta(0)),\partial U_n) \asymp 1/n^2$, and~\ref{item:convergestoboundary} and~\ref{item:seriesconvergence}
    follow. It follows from~\ref{eqn:notDW} that $\theta(\zeta)$ is not in the Denjoy--Wolff
    set for almost every $\zeta\in\partial\DD$. Since $\theta\colon \DD\to U_0$
    is a conformal isomorphism, we conclude that the Denjoy--Wolff set has zero harmonic
    measure. Since $\psi$ is conformal on $\DD$ it follows from \cite[Lemma 8.6]{benini-fagella-evdoridou-rippon-stallard2} that each $U_n$ is a contracting wandering domain. Finally, each $U_n\defeq \theta(V_n)$ is a Jordan domain with smooth boundary.
    This completes the proof.
\end{proof}

We remark that~\cite[Section~8]{benini-fagella-evdoridou-rippon-stallard2} contains
one more family of examples (Example 8.3). This example primarily concerns the behaviour
of a sequence of Blaschke products on the unit disc, and it is not clear that
its realisation by wandering domains of transcendental entire functions (which is also
possible using our theorems, in analogy to the above) is of interest.

\bibliographystyle{amsalpha}

\bibliography{bibliography}

\end{document}